\documentclass[11pt]{amsart}

\usepackage[T1]{fontenc}
\usepackage[utf8]{inputenc} 
\usepackage[main=english,french,german]{babel}

\calclayout 

\usepackage{amsmath}     
\usepackage{amsthm}      
\usepackage{amssymb}     
\usepackage{mathtools}   
\usepackage{esint}       
\usepackage{thmtools}

\usepackage{graphicx}    
\usepackage{xcolor}      
\usepackage{hyperref}    
\usepackage{enumitem}    

\newcommand{\mc}{\mathcal}
\newcommand{\N}{\mathbb{N}}  
\newcommand{\R}{\mathbb{R}}  
\newcommand{\C}{\mathbb{C}}  
\renewcommand{\a}{\alpha}
\renewcommand{\b}{\beta}
\renewcommand{\d}{\delta}
\newcommand{\e}{\varepsilon}
\newcommand{\p}{\phi}
\newcommand{\s}{\psi}
\newcommand{\w}{\omega}

\DeclareMathOperator{\spt}{supp}
\DeclareMathOperator{\E}{\mathbb{E}}     
\DeclareMathOperator{\Prob}{\mathbb{P}}  

\theoremstyle{plain}
\declaretheorem[name=Theorem,numberwithin=section]{thm}

\newtheorem{prop}[thm]{Proposition}

\theoremstyle{definition}
\newtheorem{definition}{Definition}[section]
\newtheorem{prob}{Problem}

\theoremstyle{remark}
\newtheorem{remark}{Remark}[section]

\usepackage{scalerel,stackengine}
\stackMath
\newcommand\widecheck[1]{%
  \savestack{\tmpbox}{\stretchto{%
    \scaleto{%
      \scalerel*[\widthof{\ensuremath{#1}}]{\kern-.6pt\bigwedge\kern-.6pt}%
      {\rule[-\textheight/2]{1ex}{\textheight}}%
    }{\textheight}%
  }{0.5ex}}%
  \stackon[1pt]{#1}{\scalebox{-1}{\tmpbox}}%
}

\usepackage{amsrefs}

\title{Tangency counting for well-spaced circles}
\author{Dominique Maldague}
\address{Dominique Maldague, University of Cambridge, UK and University of California, Los Angeles, USA}
\email{dm672@cam.ac.uk, dmal@math.ucla.edu}

\author{Alexander Ortiz}
\address{Alexander Ortiz, Rice University, Houston, TX, USA}
\email{ao80@rice.edu}

\date{\today}


\begin{document}

\begin{abstract}
In the late 90's, Tom Wolff introduced the circle tangency counting problem in his expository article on the Kakeya conjecture. For collections of well-spaced circles, we break the  $N^{3/2}$-barrier, proving that a set of $N$ well-spaced circles has at most $N^{25/18+\varepsilon}$ sites of internal tangency. The circle tangency problem can be related to a problem about incidences between points in $\R^3$ and light rays. For this problem, we introduce a stopping time argument to extract maximal information about well-spaced points from a refined decoupling theorem for the light cone in $\R^3$, leading to sharp bounds on the number of $\mu$-rich tangency rectangles.
\end{abstract}

\maketitle

\section{Introduction}\label{sec:intro}

\subsection{Circle tangencies and unit distances} In 1999, Wolff introduced the following discrete tangency counting problem for sets of circles in the plane.
\begin{prob}\label{prob:pairs}
    Let $\mathcal C$ be a set of $N$ circles in the plane, no three of which are internally tangent at a point. Estimate the cardinality of the set
    \[
    \mathcal T_{pair}(\mathcal C) = \{(C,C')\in\mathcal C^2:\text{$C,C'$ are internally tangent}\}.
    \]
\end{prob}
\begin{figure}[t]
    \includegraphics[width=5cm]{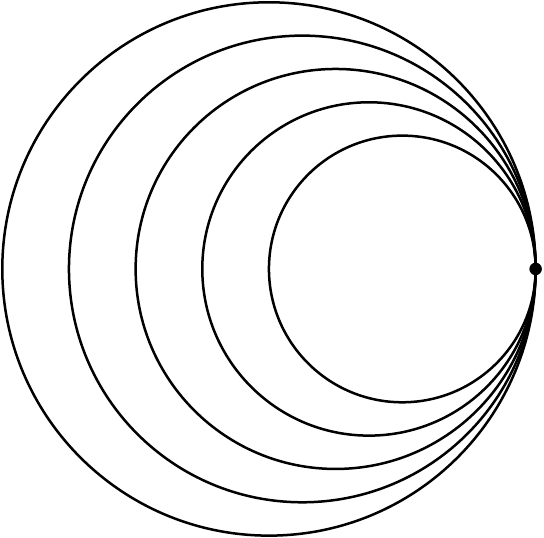}
    \caption{A clamshell of $5$ circles.}
    \label{fig:clam}
\end{figure}
The condition that no three circles are tangent at a point is a non-degeneracy condition. Without it, there are examples of collections of circles that have many tangency pairs. See Figure \ref{fig:clam} for an image of a \emph{clamshell} configuration, where every pair of distinct circles are internally tangent.

The tangency counting problem can be phrased as a problem about a set of points in $\R^3$ and \emph{light rays}, lines making an angle of $45^\circ$ with the $xy$-plane. To do this, we associate to the circle $C_{z,r}$ in the plane with center $z\in \R^2$ and radius $r > 0$ the point $(z,r)\in\R^2\times \R$. The condition that $C_{z,r},C_{z',r'}$ are internally tangent becomes the quadratic constraint:
\[
|z-z'|^2 = |r-r'|^2.
\]
Equivalently, the points $(z,r),(z',r')$ both must lie on a common light ray. In this formulation, Problem \ref{prob:pairs} is a cousin of the famous unit-distance problem in $\R^3$, first posed by Erd\H{o}s.

\vspace{1mm}\noindent\textbf{The unit distance problem: }
In $\R^3$, Erd\H{o}s' unit-distance problem asks: given a discrete set $X \subset \R^3$ with $|X| = N$, what is the maximum number of pairs $(x, x') \in X^2$ whose distance is exactly $1$?

The unit-distance requirement can be expressed as a quadratic equation 
\[
|x-x'|^2 = 1,
\]  
and the goal is an estimate for the cardinality of the set
\[
U(X) = \{(x,x')\in X^2 : |x-x'|^2 = 1\}.
\]
Problem \ref{prob:pairs} and the unit-distance problem in $\R^3$ have historically been attacked with similar algebraic--combinatorial techniques. Indeed, Wolff \cite{wolff1999recent} pointed out that the cellular partitioning technique of Clarkson, Edelsbrunner, et al. \cite{clarkson1990combinatorial}, which achieves $O_\e(N^{3/2+\e})$ unit-distances in $\R^3$, can be adapted in a straightforward way to the formulation of Problem \ref{prob:pairs} in terms of points in $\R^3$ and light rays to prove that $|\mathcal T_{pair}(\mathcal C)| = O_\varepsilon(N^{3/2+\varepsilon})$ for collections $\mathcal C$, no three of which are tangent at a point. Both bounds are believed to be too big, but have been hard to improve. More recently, Ellenberg, Solymosi, and Zahl used a different approach in the tangency counting problem, removing the $\varepsilon$, showing $|\mathcal T_{pair}(\mathcal C)| = O(N^{3/2})$ \cite{ellenberg2016new}.

In \cite{zahl2019breaking}, Zahl used a more refined approach to the unit-distance problem, combining ideas from \cite{ellenberg2016new} with the method of polynomial partitioning to get a small improvement in the exponent of the unit-distance problem in $\R^3$, showing $|U(X)| = O_\e(N^{295/197+\e})$. For the circle tangency problem, Problem \ref{prob:pairs}, $N^{3/2}$ remains the bound to beat to this day. In this paper, we break the $3/2$-barrier for Problem \ref{prob:pairs} for collections of circles that are \emph{well-spaced}. We introduce the setup for well-spaced circles now. 

\begin{definition}\label{def:well-spaced}
    If $Q\subset\R^n$ is a unit cube, then a subset $X\subset Q$ is \emph{well-spaced} if $X$ is $\rho$-separated for some $\rho\in(0,1)$, and moreover, $X$ is maximal with respect to this property. In particular, $c_n\rho^{-n}\le|X|\le C_n\rho^{-n}$ for dimensional constants $0<c_n<C_n$.
\end{definition}
\begin{thm}\label{thm:well-spaced-tang}
    For every $\varepsilon>0$, there is a constant $A_\e$ so that the following holds.  
Let $X\subset [0,1]^2\times[1,2]$ be well-spaced, and let $\mathcal C_X$ denote the collection of circles
\[
\{C_{z,r}:(z,r)\in X\},
\]
where $C_{z,r}$ denotes the circle with center $z\in\R^2$ and radius $r>0$.

If no three circles of $\mathcal C_X$ are tangent at a point, then the following estimate holds:
\[
|\mathcal T_{pair}(\mathcal C_X)|\le  A_\e|X|^{25/18+\varepsilon}.
\]
\end{thm}
Theorem \ref{thm:well-spaced-tang} will be deduced from Theorem \ref{thm:rect1}, a sharp result for a continuum version of the problem we describe momentarily.

In addition to upper bounds for circle tangencies or unit-distances, one could ask about lower bounds. In this regard, well-spaced sets are important because they give the current best known lower bounds for both problems. In \cite{erdos1960sets}, Erd\H{o}s showed that amongst the points in a $N^{1/3}\times N^{1/3}\times N^{1/3}$ integer lattice, the most popular distance (and hence by a rescaling, the unit distance) appears at least $cN^{4/3}$ times. In \cite{schlag2003continuum}, Schlag showed that with $X$ equal to a $N^{1/3}\times N^{1/3}\times N^{1/3}$ lattice inside $[0,1]^2\times [1,2]$, the set $\mathcal C_X$ of circles has at least $cN^{4/3}$-many pairs that are internally tangent. Both these examples are well-spaced in the sense of Definition \ref{def:well-spaced}.

\begin{figure}
\includegraphics[width=0.5\textwidth, trim= 80 70 70 80, clip]{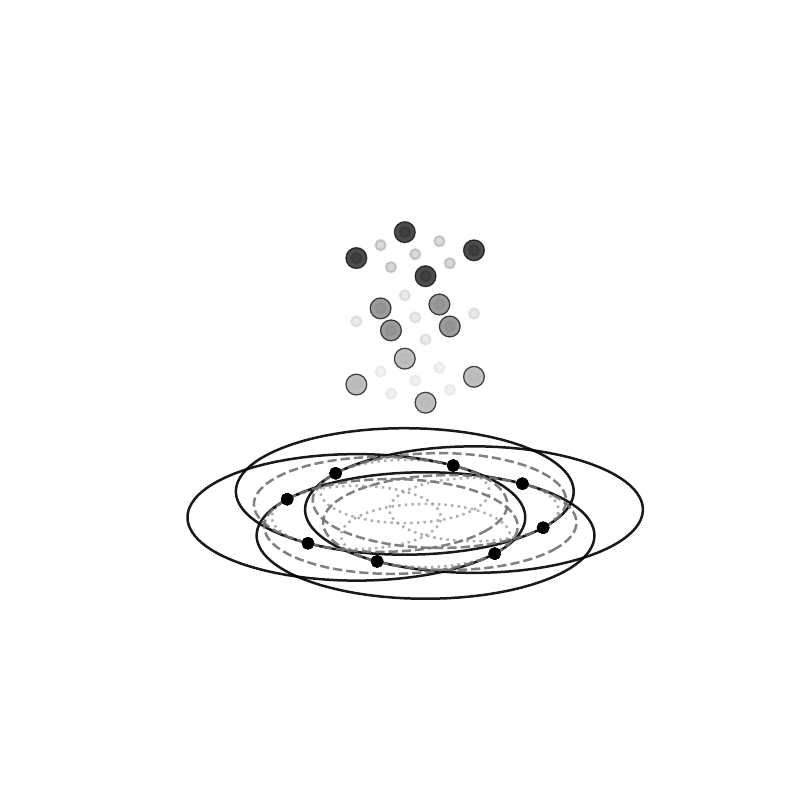}
    \caption{A lattice $X$ (points above) and a selected subset of circles $\mathcal{C}_X$ (below), with black dots marking the sites of internal tangencies between circles.}
    \label{fig:well-spaced}
\end{figure}

See Figure \ref{fig:well-spaced} for an illustration of the grid $X = (\frac{1}{2}\mathbb Z)^3\cap [0,1]^2\times[1,2]$ and a selection of some of the corresponding circles in the plane.

\subsection{Discrete versus continuum problems}
Wolff introduced Problem \ref{prob:pairs} as a discrete analog of the following Kakeya-type problem for circles. Define a BRK set (after Besicovitch--Rado \cite{besicovitch1968plane} and Kinney \cite{kinney1968thin} who initiated the study of these sets) to be a compact set $E\subset\R^2$ containing circles of every radius in the interval $[1,2]$. Every BRK set is a $1$-parameter family of $1$-dimensional objects (circles in this case), so whether such a set has dimension $2$ can be seen as a borderline question.

By taking the $\delta$-neighborhood $\mathcal N_\delta(E)$ of a BRK set $E$, determining the dimension of $E$ reduces to the geometry problem of understanding how thin neighborhoods of circles overlap. If two $\delta$-annuli overlap ``transversely,'' then their region of overlap is contained in the union of $O(1)$-many $\delta\times\delta$-squares. If two $\delta$-annuli overlap ``tangentially,'' then that means both annuli contain a common rectangle of dimensions approximately $\delta\times\sqrt\delta$. Intuitively, for a set of circles to have small dimension, many pairs of circles should overlap tangentially, minimizing how much the set ``spreads out.''

Wolff gave two proofs that the dimension of a BRK set is $2$. To describe one of his approaches, we introduce a variant of Problem \ref{prob:pairs} that focuses on the points of tangency for a collection $\mathcal C$, rather than the number of pairs of tangent circles.

\begin{prob}[Counting tangency sites]\label{prob:tangency-sites}
    If $\mathcal C$ is a collection of $N$ circles in $\R^2$, determine the number of sites of internal circle tangency, i.e., the cardinality of the set
    \[
    \mathcal T(\mathcal C) = \{z\in\R^2 : \text{at least two circles of $\mathcal C$ are internally tangent at $z$}\}.
    \]
    See Figure \ref{fig:circle-tangencies} for an illustration of a collection of $6$ circles with $6$ sites of internal tangency.
\end{prob}
\begin{figure}
    \centering
    \includegraphics[width=7cm]{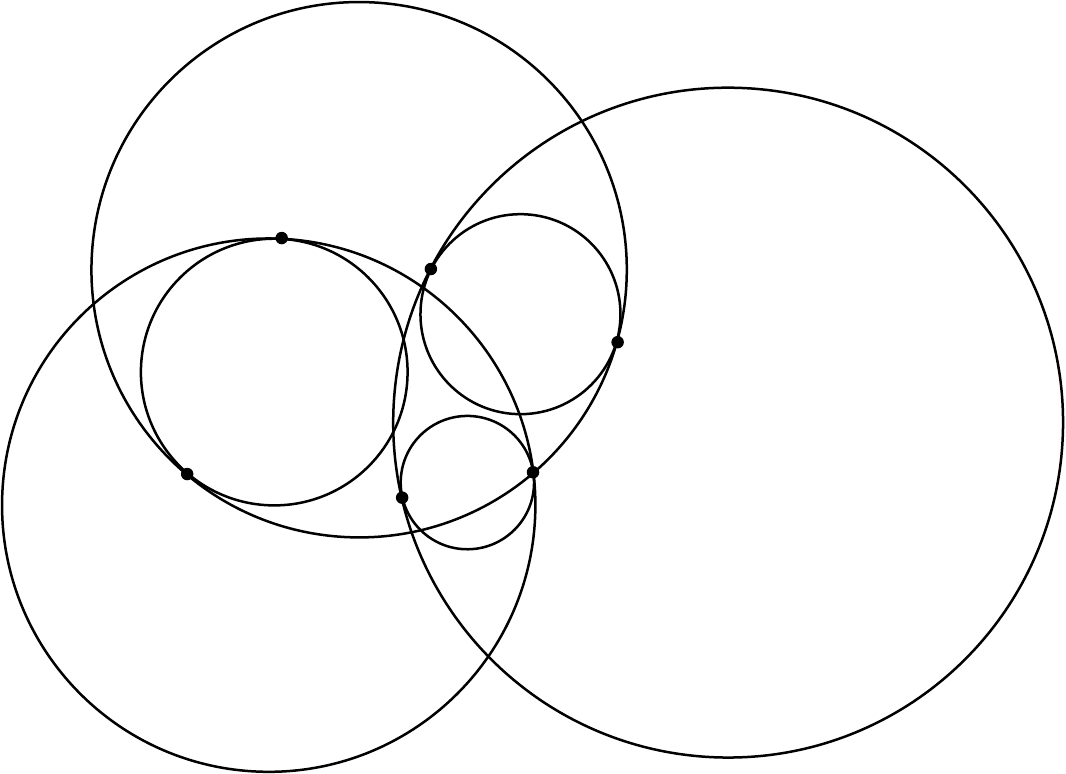}
    \caption{A set of $6$ circles with $|\mathcal T(\mathcal C)| = 6$.}
    \label{fig:circle-tangencies}
\end{figure}

If $\mathcal C$ is a family of circles, no $3$ of which are tangent at a point, then actually $|\mathcal T_{pair}(\mathcal C)| = |\mathcal T(\mathcal C)|$, so without the assumption that no three circles are tangent at a point, Problem \ref{prob:pairs} can be seen as a version of Problem \ref{prob:tangency-sites} with multiplicity. The current best bound for Problem \ref{prob:tangency-sites} is $|\mathcal T(\mathcal C)| = O(N^{3/2})$ by the work of Ellenberg--Solymosi--Zahl, and the conjecture is that $|\mathcal T(\mathcal C)| = O(N^{4/3})$ on the basis of the lattice example.

In his paper ``Local smoothing type estimates on $L^p$ for large $p$'' \cite{wolff2000local}, Wolff addressed a continuum analog of Problem \ref{prob:tangency-sites} using the cellular partitioning technique of Clarkson, Edelsbrunner, et al.

\begin{prob}\label{prob:rect}
    Let $X\subset[0,1]^2\times[1,2]$ be such that for every $x\ne x'\in X$, we have
    \[
    d(x,x') \ge \delta,
    \]
    for a small parameter $\delta>0$ (so $X$  is a $\delta$-separated set), and consider the collection $\mathcal C_X$ of circles defined by
    \[
    \mathcal C_X = \{C_x : x\in X\}.
    \]
    Say a $\delta\times\sqrt\delta$-rectangle $\Omega$ in the plane is a $\mu$-rich tangency rectangle for $\mathcal C_X$ if 
    \[
    \Omega\subset \mathcal N_{10\delta}(C_{x})
    \]
    holds for at least $\mu$ points $x\in X$. (The notation $\mathcal N_r(E)$ is used for the open $r$-neighborhood of a set $E$.)

    Say a collection $\mathcal R$ of $\delta\times\sqrt\delta$-rectangles in the plane is \emph{pairwise incomparable} if no two rectangles of the collection are contained in the $10$-fold dilation of the other about their centers.
    
    \textbf{Problem:} Estimate the maximal cardinality of a collection $\mathcal R_\mu$ of pairwise incomparable $\mu$-rich tangency rectangles for $\mathcal C_X$.
\end{prob}

We digest the statement of Problem \ref{prob:rect} a bit. Without the assumption that $X$ is $\delta$-separated, every circle of $\mathcal C_X$ is allowed to be an infinitesimal perturbation of a single fixed circle, and then there are approximately $\delta^{-1/2}$-many $\delta\times\delta^{1/2}$-rectangles contained in the $10\delta$-neighborhood of the fixed circle, regardless of the cardinality of $X$.

Heuristically, Wolff proved\footnote{Actually, Wolff proved a bound for the following bipartite variant of Problem \ref{prob:rect}. Given $X,Y\subset[0,1]^2\times[1,2]$ with $d(X,Y)\ge 0.1$, each of which is $\delta$-separated, and a pairwise incomparable collection $\mathcal R_{\ge\mu,\ge\nu}$ of $\delta\times\sqrt\delta$-rectangles $\Omega$ satisfying: for at least $\mu$ points $x\in X$ and at least $\nu$ points $y\in Y$, $\Omega\subset\mathcal N_{10\delta}(X)\cap\mathcal N_{10\delta}(Y)$, one has the bound $$|\mathcal R_{\ge \mu,\ge\nu}| \le C_\e\delta^{-\e}\left((\frac{|X||Y|}{\mu\nu})^{3/4} + \frac{|X|}{\mu} + \frac{|Y|}{\nu}\right).$$} that if $\mathcal R_\mu$ is a pairwise incomparable collection of $\mu$-rich $\delta\times\sqrt\delta$-tangency rectangles, one has the bound
\begin{equation}\label{eq:wolff-rect}
    |\mathcal R_\mu| \le C_\e\delta^{-\e}(\frac{|X|}{\mu})^{3/2},
\end{equation}
which we should think of as the analog of the $O_\e(N^{3/2+\e})$-bound for Problem \ref{prob:tangency-sites}. As a consequence of his estimate for tangency rectangles, Wolff proved a sharp \emph{decoupling} estimate for the truncated light cone
\[
\Gamma = \{(\xi_1,\xi_2,\xi_3) \in \R^3 : \xi_3 = \sqrt{\xi_1^2 + \xi_2^2}, 1\le \xi_3\le 2\}.
\]
Decoupling (also known as Bourgain--Demeter decoupling \cite{bourgain2015proof}) for the light cone, is a phenomenon in Fourier analysis which is a measurement of orthogonality between different frequency localized pieces of a solution $u$ to the wave equation in $L^p$ for $p > 2$, and Wolff used the sharp decoupling inequality for the cone in $L^{74+\e}$ to prove optimal $L^{74}$ local smoothing bounds for the wave equation in $2+1$ dimensions. These local smoothing bounds for the light cone imply that the dimension of a BRK set is $2$ \cite{schlag1996thesis}.

We prove a sharp bound for the number of $\mu$-rich tangency rectangles for collections of well-spaced circles, which apart from $\e$ is the tangency rectangle analog of the conjectured $O(N^{4/3})$ bound for Problem \ref{prob:tangency-sites} for well-spaced circles.
\begin{thm}\label{thm:rect1}
    For every $\varepsilon>0$, there is a constant $A_\varepsilon$ so the following holds for every well-spaced $X\subset [0,1]^2\times[1,2]$.  For any $\mu \ge 1$, and $\tau\in[|X|^{-1/3},1)$, let $\mathcal R_{\mu,\tau}$ be an arbitrary pairwise incomparable set of $\mu$-rich $\tau^2\times\tau$-rectangles for the collection $\mathcal C_X$. Then the following estimate holds:
    \[
    \mu^{4/3}|\mathcal R_{\mu,\tau}|\le
     A_\e|X|^{4/3+\e}.
    \]
\end{thm}
The estimate of Theorem \ref{thm:rect1} is sharp in the sense that for any $N>1$, there is a well-spaced set $X\subset[0,1]^2\times[1,2]$ of cardinality $N$ with the following property. For each $\tau\in[N^{-1/3},1)$, there exists $\mu \ge 1$ such that every $\tau^2\times\tau$-rectangle in a maximal pairwise incomparable collection $\mathcal R_{\mu,\tau}$ is approximately $\mu$-rich, and for which we have the lower bound
\[
\mu^{4/3}|\mathcal R_{\mu,\tau}| = \Omega_\e( N^{4/3-\e}).
\]
(See Theorem \ref{thm:sharp}.) Extending the estimate $\mu^{4/3}|\mathcal R_{\mu,\tau}|\le A_\e|X|^{4/3+\e}$ of Theorem~\ref{thm:rect1} to the range $\tau \in (0, |X|^{-1/3})$ would yield corresponding improvements to the exponent in Theorem~\ref{thm:well-spaced-tang} for well-spaced circles. Such an extension requires new ideas, and we plan to investigate this range of $\tau$ in future work.

\subsection{Decoupling and the proof of Theorem~\ref{thm:rect1}} Our proof of Theorem~\ref{thm:rect1} is based entirely on Fourier analysis and the lifting procedure that replaces circles by points in $\R^3$. We rely on an extension of this basic correspondence due to Wolff \cite{wolff2000local}. If $\Omega$ is a $\delta\times\sqrt\delta$-rectangle, and $C_{z,r}$ is a circle whose $10\delta$-neighborhood contains $\Omega$, then there is a $\delta\times\sqrt\delta\times 1$-box $P$ in $\R^3$ which is tangent to a light cone, and intersects the $10\delta$-ball centered at $(z,r)\in\R^2\times\R$. We refer to such a box in $\R^3$ as a \emph{lightplank} or a \emph{cone plank}. We say a lightplank is $\mu$-rich for a collection $X\subset[0,1]^2\times[1,2]$ if its $C\delta$-neighborhood contains at least $\mu$ points of $X$. For more details on the correspondence between tangency rectangles and lightplanks, we refer to the article \cite{ortiz2024sharp} of the second author.

Lightplanks are significant because in the restriction theory of the light cone, a \emph{wave packet} is essentially supported in a lightplank. The duality between tangency rectangles and lightplanks allows us to count $\mu$-rich tangency rectangles by counting the corresponding $\mu$-rich lightplanks, for which tools from restriction theory become available. If $\mathcal P_\mu$ is a pairwise incomparable collection of $\mu$-rich lightplanks for $X$, we can form a function $f = \sum_{P\in\mathcal P_\mu} \phi_P$, where each function $\phi_P$ is a smooth approximation of the indicator function of a lightplank $P$, but with Fourier transform $\widehat\phi_P$ supported in a dual lightplank of dimensions $1\times\delta^{-1/2}\times\delta^{-1}$ centered at the origin. 

The Fourier transform $\widehat f(\xi)$ is supported in a union of $1\times \delta^{-1/2}\times\delta^{-1}$-lightplanks containing the origin. We study $f$ using the well-known technique of high--low frequency analysis and consider the various frequency components of $f$, roughly corresponding to the regions $\{\frac{\delta^{-1}}{2K} < |\xi_3| <\frac{\delta^{-1}}{K}\}$ as $K$ varies. See the papers \cites{cohen2025lower,guth2019incidence,gan2022restricted,ren2023furstenberg,gan2024exceptional} for some select examples of this approach to problems in continuum incidence geometry, projection theory, and geometric measure theory. For our purposes, we need a more refined approach than has been considered previously. In particular, our high--low analysis relies on a version of \emph{refined decoupling} involving square functions of $f$ (Theorem \ref{Grefdec}).

Briefly, we consider the local behavior of the function $f$ near each point $x\in X$, and we use a stopping time algorithm to sift out the dominant frequency component of $f$ which is \emph{active} near $x$. Once we have localized $f$ in frequency near $x$, we apply our refined decoupling theorem for the light cone formulated in terms of square functions (Theorem~\ref{Grefdec}). The stopping time argument we use in the proof of Theorem \ref{thm:rect1} echoes the structure of square function estimates in the spirit of the Burkholder--Davis--Gundy inequality \cite{burkholder1972integral}. This viewpoint does not seem to have been widely explored in the Fourier analytic context. We believe this strategy, especially the use of refined decoupling as it is formulated in Theorem \ref{Grefdec} in terms of square functions, will be useful in approaching other continuum incidence geometry problems as well.

\subsection{Outline of the paper} In Section~\ref{sec:rectangles}, we describe the lifting procedure introduced by Wolff in~\cite{wolff2000local}, which translates $\delta$-close circle tangencies into point–lightplank incidences, and we prove our main estimate Theorem~\ref{thm:rect1}, as well as Theorem~\ref{thm:well-spaced-tang}.
In Section \ref{sec:decoupling} we prove a version of square function refined decoupling theorem for the light cone in $\R^3$, which is used as input to the proof of Theorem~\ref{thm:rect1}. In Section~\ref{sec:sharpness}, we illustrate the sharpness of Theorem~\ref{thm:rect1} by constructing a random well-spaced set of circles obeying the assumptions of Theorem~\ref{thm:rect1}, for which there is only one value of $\mu$.

\subsection{Notation}

We write \( \mathcal{N}_r(E) \) for the Euclidean \( r \)-neighborhood of a set \( E \subset \mathbb{R}^n \).

If \( A \leq C B \) for an absolute constant \( C \), we write \( A \lesssim B \); if both \( A \lesssim B \) and \( B \lesssim A \) hold, we write \( A \sim B \). When the implicit constant depends on parameters such as \( \varepsilon \), we write \( A \lesssim_\varepsilon B \) to emphasize the dependence.

We use the notation \( A \lessapprox B \) to mean that for every \( \varepsilon > 0 \), we have \( A \lesssim_\varepsilon \delta^{-\varepsilon} B \), where \( \delta > 0 \) is a small parameter. Similarly, when \( R > 1 \), the notation \( A \lessapprox B \) means \( A \lesssim_\varepsilon R^\varepsilon B \) for all \( \varepsilon > 0 \). In both cases, we write \( A \approx B \) when the corresponding two-sided bound holds.

We write \( \mathrm{RapDec}(f) \) for any quantity that is \( \lesssim_k R^{-k} \|f\|_\infty \) for all \( k \in \mathbb{N} \), where the implicit constant may depend on \( k \), but not on \( f \) or \( R \).

\subsection{Acknowledgments} We are grateful to Larry Guth for his thoughtful advice and for insightful discussions that enriched our understanding of this topic, as well as his constructive feedback on an earlier version of this work.

\section{Estimate of \texorpdfstring{$\mu$}{mu}-rich tangency rectangles}\label{sec:rectangles}

In this section, we will prove Theorem \ref{thm:rect1}. As a corollary, we will obtain Theorem~\ref{thm:well-spaced-tang}, breaking the 3/2-barrier on Problem \ref{prob:tangency-sites} for well-spaced circles.

We employ a lifting technique, originally introduced by Wolff, to convert $\mu$-rich tangency rectangles into $\mu$-rich lightplanks for the set of points $X \subset \mathbb{R}^3$ that encodes our collection of circles. We then apply a stopping time argument, in combination with a new refined decoupling theorem for the light cone (Theorem~\ref{Grefdec}), to establish the main estimate.

Throughout, we will identify points of $X
\subset\mathbb{R}^3$ with their corresponding circles in the plane with given center-radius pairs. For $x\in X$, written in coordinates as $(\bar x,x_3)\in [0,1]^2\times[1,2]$, let
\[
C_{\delta,x} = \{z\in \R^2:||z-\bar x|-x_3|<\delta\}
\]
be the $\delta$-thick annulus with core circle $x$. If $x,x'\in X$ and $\Delta(x,x')<\delta$, then $C_{10\delta,x}\cap C_{10\delta,x'}$ contains a $\delta\times\sqrt\delta$-tangency rectangle in the plane.

We adopt the notation and terminology from the second author’s article~\cite{ortiz2024sharp}, 
where Wolff's lifting procedure—translating circle tangencies into equivalent statements 
about points and lightplanks—is described in detail.  

Let $\Omega\subset\R^2$ be a $\delta\times\sqrt\delta$ rectangle, and set 
\[
Q := [0,1]^2\times[1,2].
\]
Define  
\[
\mathbf D_{10\delta}(\Omega) \;:=\; \{ x\in Q : \Omega \subset C_{10\delta,x} \},
\]
the set of points in $Q$ whose associated annulus $C_{10\delta,x}$ contains $\Omega$.  

If 
\[
|X\cap \mathbf D_{10\delta}(\Omega)| \gtrsim \mu
\quad\text{(or $\sim \mu$)},
\]
we say that $\Omega$ is a \emph{$\mu$‑rich tangency rectangle} (respectively, \emph{$\sim\mu$‑rich}). When $\Omega$ is a $\delta\times\sqrt\delta$ rectangle, 
$\mathbf D_{10\delta}(\Omega)$ is a $\sim1\times\sqrt\delta\times\delta$ lightplank.

\begin{definition}\label{def:lightplank}
    Let $\gamma\colon[-\pi,\pi)\to \mathbb R^3$ be the unit-speed curve defined by $\gamma(\theta) = \frac{1}{\sqrt 2}(\cos\theta,\sin\theta,1)$. If for some $A,B>0$, $\theta_P\in[-\pi,\pi)$ and $v\in\mathbb R^3$,
    \[
    P = v + \{a\gamma(\theta_P) + b\gamma'(\theta_P) + c(\gamma\times\gamma')(\theta_P): |a|\le \frac{A}{2},|b|\le \frac{\sqrt{AB}}{2},|c|\le \frac B2\},
    \]
    then we say $P$ is a \emph{$A\times \sqrt{AB}\times B$-lightplank (or cone plank)} with \emph{center} $v$.
\end{definition}

By rescaling, letting $R:= 1/\delta$, instead of considering $\delta$-neighborhoods of circles with center-radius pairs in $[0,1]^2\times[1,2]$, we will consider the unit-neighborhoods of circles with center-radius pairs in the cube $[0,R]^2\times[R,2R]$. This is done to match the notation more common to restriction theory, where we consider functions with frequency support in the unit ball, and physical scale features in a large ball of radius $R> 1$. 

We say two $1\times \sqrt R\times R$-lightplanks are \emph{$K$-incomparable} if neither is contained in the $K$-dilation of the other. Often we will just drop the ``$K$'' and say that two lightplanks are incomparable.

We let $\Gamma = \{(\xi_1,\xi_2,\xi_3)\in\R^3:\xi_3 = \sqrt{\xi_1^2 + \xi_2^2}, 1 \le \xi_3 \le 2\}$ be the usual segment of the light cone. Our proof of Theorem \ref{thm:rect1} will be based on the high--low method. Before giving the full argument, we give a sketch of the proof of the high-frequency dominated case.

\vspace{1mm}\noindent\textbf{A sketch of the high-frequency case:}
Suppose $X$ is maximally $\sqrt R$-separated in $[0,R]^2\times[R,2R]$, and let $\mathcal P_{\mu}$ be a pairwise incomparable collection of $\mu$-rich $1\times \sqrt R\times R$-lightplanks for $X$. Our goal is to show
\[
\mu^{4/3}|\mathcal P_\mu| \lessapprox |X|^{4/3}.
\]
Define $f = \sum_{P\in\mathcal P_\mu}\phi_P$ where each function $\phi_P$ is a smooth approximation of the indicator function of a lightplank $P$, but with Fourier transform $\widehat\phi_P$ supported in a dual lightplank of dimensions $1\times R^{-1/2}\times R^{-1}$ centered at the origin.

The starting point of our estimate is the calculation of the incidences between the disjoint unit-balls of the $1$-neighborhood $\mathcal N_1 (X)$ and the $\mu$-rich lightplanks in $\mathcal P_\mu$:
\begin{equation}\label{eq:incidences}
\mu|\mathcal P_\mu| \lesssim \int_{\mathcal N_1(X)} f.
\end{equation}
We assume $f = f_{high}$, where
\[
f_{high}(x) = \int_{|\xi_3|\ge 1} \widehat f(\xi) e^{2\pi ix\cdot\xi}\,d\xi
\]
is the \emph{high-frequency} part of $f$. The high-frequency part of $f$ is Fourier-supported in the $R^{-1}$-neighborhood of the truncated light cone
\[
\{(\xi_1,\xi_2,\xi_3):\xi_3^2 = \xi_1^2 + \xi_2^2, 1\le |\xi_3|\le 2\},
\]
where decoupling estimates can be applied. This is the motivation for considering the high--low method in our problem. By dyadic pigeonholing, we can assume that every ball of $\mathcal N_1(X)$ is contained in exactly $M$ lightplanks of $\mathcal P_\mu$, for some $M\in[1,R^{1/2}]$. In this scenario, we can apply refined decoupling for the light cone, which gives an improved bound over Bourgain--Demeter's decoupling inequality when the region in physical space we integrate over only receives contributions from a few wave packets (see Theorem A.1 from Harris' article for a clean statement that applies here \cite{harris2024projections}).

Since we assume every ball of $\mathcal N_1(X)$ is $M$-rich, and $|\mathcal N_1(X)| \sim |X|$, from \eqref{eq:incidences} we obtain
\begin{equation}\label{eq:double-count}
\mu|\mathcal P_\mu| \lessapprox M|X|.
\end{equation}
Refined decoupling is an $L^6(X)$ estimate. By refined decoupling, and the calculation $\|\phi_P\|_{6}^6\sim |P| = R^{3/2}$,
\begin{align*}
    M^6|X|\sim\int_X|f|^6 \lessapprox (\frac{M}{|\mathcal P_\mu|})^2(\sum_{P\in\mathcal P_\mu}\|\phi_P\|_6^2)^3 \sim (\frac{M}{|\mathcal P_\mu|})^2 |\mathcal P_\mu|^3 R^{3/2}.
\end{align*}
Rearranging and combining like terms, we have
\[
M^4|X|\lessapprox |\mathcal P_\mu|R^{3/2}.
\]
Applying \eqref{eq:double-count} (to lower bound $M^4\ge (\mu|\mathcal P_\mu|/|X|)^4$), and combining like terms, we get
\[
\mu^4|\mathcal P_\mu|^3 \lessapprox R^{3/2}|X|^3.
\]
Finally, by the well-spaced assumption, $|X| \sim R^{3/2}$, so we confirm
\[
\mu^4|\mathcal P_\mu|^3\lessapprox |X|^4,
\]
and taking $3$rd-roots yields the desired bound.\hfill$\square$

\vspace{1em}



We have sketched the proof of the estimate $\mu^{4/3}|\mathcal P_\mu| \lessapprox |X|^{4/3}$ in the high–frequency dominated case, where the Fourier support of $f$ is essentially the $R^{-1}$-neighborhood of $\Gamma$, and refined cone decoupling is strongest. If $f = f_{K}$ dominates the integral $\int_{\mathcal N_1(X)}f$, where $K\gg 1$, and
\[
f_K(x) = \int_{|\xi_3|\sim 1/K}\widehat f(\xi)e^{2\pi ix\cdot\xi}\,d\xi,
\]
then the bounds we get from refined decoupling are a priori weaker. To compensate for this, we use the observation that $f_K$ is locally constant on balls of radius $K$ (Proposition \ref{prop:loc-const}), and the separation assumption of $X$ to get an additional gain; namely that
\[
\int_{\mathcal N_1(X)} f_K \approx \frac{1}{K^3}\int_{\mathcal{N}_K(X)}f_K,
\]
which holds as long as $K$ is less than the separation between distinct points of $X$. If $f$ is dominated by even lower frequencies, then a short direct argument yields the desired bound, and we will treat this case separately at the end of the proof of Theorem \ref{thm:main}.

Another wrinkle in making the argument we sketched in the high--case work is that we assumed every point of $X$ was $M$-rich in the lightplanks $\mathcal P_\mu$. What we will do instead is for each $x\in X$, determine the dominant frequency component $f_{K(x)}$ near $x$ to begin, and dyadically pigeonhole so that $K(x)$ is constant for every $x\in X$.  To execute this strategy, we will define a family of averages of $f$ which are controlled by certain square functions of $f$ in the same spirit as the Burkholder--Davis--Gundy square function estimate. We will set up the estimate precisely in order to apply refined decoupling in the form of Theorem \ref{Grefdec}.

\begin{thm}\label{thm:main}
    For every $\varepsilon>0$, there is a constant $C_\varepsilon$ so the following holds for every $R\ge 1$.
    
    Suppose $X\subset[0,R]^3$ is a disjoint union of unit balls which are $1\le \rho\le R^{1/2}$-separated and such that $|X| \sim (R/\rho)^3$. For every $\mu\in[1,R^{3/2}]$, if $\mathcal P_\mu$ is an arbitrary pairwise incomparable collection of $\mu$-rich $1\times R^{1/2}\times R$-lightplanks for $X$, then
    \[
    \mu^{4/3} |\mathcal P_\mu|\le C_\varepsilon R^{\varepsilon} |X|^{4/3}.
    \]
\end{thm}

\begin{proof}
We may assume without loss of generality that $R\ge R_0(\e)$ is sufficiently large depending on $\e$.

For each $\sim 1\times R^{-1/2}\times R^{-1}$ cone plank $\theta$ in a partition of $\mc{N}_{R^{-1}}(\Gamma)$, write $\mc{P}_{\mu,\theta}$ for the subcollection of planks $P$ in $\mc{P}_\mu$ for which the corresponding $\theta_P$ from Definition \ref{def:lightplank} lies in the angular sector corresponding to $\theta$. Define $f_\theta=\sum_{P\in\mc{P}_{\mu,\theta}} \phi_P$, where for each $P\in\mathcal P_\mu$, $\phi_P\ge \mathbf 1_P$ is a Schwartz function with Fourier support in $\theta-\theta\subset B(0,2)$.  Some properties of $f$ we shall use are as follows:
\begin{enumerate}
    \item $\|f\|_{L^1(\R^3)} \sim \mathrm{vol}(P)|\mathcal P_\mu|=R^{3/2}|\mathcal P_\mu|$
    \item $\mu|\mathcal P_\mu|\lesssim\int_X f$.
\end{enumerate}
Consider the following martingale-like sequence constructed from $f$, a sequence of scales $\{\rho_k\}$, and a Fourier decomposition of $B(0,2)$. On a first read, we recommend reading $\mathcal A_kf(x)$ as $\mathbb E_{\rho_k}f(x) := \frac{1}{|B(x,\rho_k)|}\int_{B(x,\rho_k)}f$, which is morally the same as $\mathcal A_kf$ up to a small error.

Let $\eta$ be a radial bump function supported in $B(0,1)$, satisfying $\eta\equiv 1$ on $B(0,1/2)$. Let $\eta_r(x)=\eta(x/r)$ be the $r$-ball localized version of $\eta$, satisfying $\eta_r\equiv 1$ in $B(0,r/2)$. Assume that $R,\rho\in2^{\mathbb{Z}}$. For $R=2^m$, let $\rho_k=2^{k\lceil{m\e}\rceil}$. Let $N$ satisfy $\rho_{N-1}<\rho\le \rho_N$, so $N\lesssim \e^{-1}$, in particular. For $\rho^{-1}\le \rho_k^{-1}\le 1$, write $\w_k(x)=\frac{\rho_k^{-3}}{(1+|\rho_k^{-1}x|^2)^{\e^{-1}}}$ and note that $|\widecheck{\eta}_{\rho_k^{-1}}|(x)\lesssim_\e \w_k(x)$. We define a version of $\w_k$ which has nice Fourier support and is uniformly pointwise comparable to $\w_k$. 

Let \(\tilde{\omega}_k(x)\) be defined by  
\[
\tilde{\omega}_k(x) = \rho_k^{-3}|\widecheck{\eta}|^2(\rho_k^{-1}x) + \rho_k^{-3}\sum_{j > 0} \frac{|\widecheck{\eta}|^2(2^{-j}\rho_k^{-1}x)}{2^{2\varepsilon^{-1}j}}.
\]
It is immediate from the definition that \(\tilde{\omega}_k\) satisfies the following properties:

\begin{enumerate}[label=\textbf{(W\arabic*)}]
    \item \label{W1} \textbf{Frequency localization.}  
    The Fourier support of \(\tilde{\omega}_k\) is contained in \(B(0, 2\rho_k^{-1})\). This follows from the identity \(|\widecheck{\eta}|^2 = \widecheck{\eta}^2\), so that \(\widehat{|\widecheck{\eta}|^2} = \eta * \eta\), which is supported in \(B(0,2)\).

    \item \label{W2} \textbf{Spatial decay comparable to \(\omega_k\).}  
    For \(|x| \lesssim \rho_k\), we have \(\tilde{\omega}_k(x) \sim \omega_k(x) \sim \rho_k^{-3}\). More generally, if \(C > 1\) with \(C \in 2^{\mathbb{Z}}\) and \(|x| \sim C \rho_k\), then  
    \[
    \tilde{\omega}_k(x) \sim \omega_k(x) \sim_\varepsilon \rho_k^{-3} C^{-2\varepsilon^{-1}}.
    \]
    In particular, \(\tilde{\omega}_k(x) \sim \omega_k(x)\) for all \(x \in \mathbb{R}^3\).
\end{enumerate}

For each $x\in\R^3$, we will define a stopping time $\mathbf k(x)$ which measures the first time at which the local behavior of $f$ goes from being ``crowded'' to ``sparse'' in transition from scales $\rho_{k-1}$ to $\rho_{k}$. Intuitively,
\begin{equation}\label{eq:intuition}
f(x) \approx \mathbb E_{\rho_1}f(x)\approx\dots\approx\mathbb E_{\rho_{\mathbf k(x)-1}}f(x)\gg \mathbb E_{\rho_{\mathbf k(x)}}f(x).
\end{equation}
The condition \eqref{eq:intuition} also morally ensures that we have control over the Fourier support of $f$: 
\[
f(x)\approx \mathbb E_{\rho_{\mathbf k(x)-1}}f(x) \approx (\widehat{f}\cdot\mathbf 1_{\{\rho_{\mathbf k(x)}^{-1}<|\xi|<\rho_{\mathbf k(x)-1}^{-1}\}})^{\vee}(x).
\]
To make this rigorous, and to define the stopping time $\mathbf{k}(x)$ precisely, we will use the weights $\tilde{\omega}_k$. Since
\[
f = \sum_\theta f_\theta \quad \text{and} \quad \operatorname{supp}(\widehat{f_\theta}) \subset \theta-\theta\subset B(0,2),
\]
we have the pointwise estimate
\[
f(x) \lesssim_\e f*\tilde\w_0(x).
\]

Fix $x \in B_R$. We now describe a \textbf{stopping time algorithm} that iteratively measures decay of the localized averages of $f$ at $x$. Define the initial average
\[
\mathcal{A}_0 f(x) := \sum_\theta f_\theta * \tilde\w_{0}(x).
\]
For each $k \geq 1$, define
\[
\tilde{\omega}_{k,\theta} := |\det T_{k,\theta}|\,\tilde{\w}_0\circ T_{k,\theta},
\]
where $T_{k,\theta}$ is an affine transformation, and $\tilde{\omega}_{k,\theta}$ is spatially localized to an origin-centered $\rho_k \times R^{1/2} \times R$ stack of $1\times R^{1/2}\times R$ dual planks to $\theta$. To align with the notation when $k = 0$, we define $\tilde\w_{0,\theta}=\tilde\w_0$ for every $\theta$.

For $k \ge 1$, set
\[
\mathcal{A}_k f := \sum_\theta f_\theta * \tilde\w_{0,\theta} * \tilde{\omega}_{1,\theta} * \cdots * \tilde{\omega}_{k,\theta}.
\]
Since $k \lesssim\e^{-1}$, 
\[
\tilde \w_{0,\theta}*\tilde{\omega}_{1,\theta} * \cdots * \tilde{\omega}_{k,\theta}(x) \sim_\e \tilde\omega_{k,\theta}(x).
\] 
Each $\tilde{\omega}_{k,\theta}$ is Fourier supported in the ball $B(0, 2\rho_k^{-1})$, and consequently, $\mathcal{A}_k f$ is also Fourier supported in this ball. Note that \(\mathcal{A}_{k-1}f\) and \(\mathcal{A}_k f\) differ by convolution with \(\tilde\omega_{k,\theta}\), where the somewhat intricate form of \(\mathcal{A}_k f\) is chosen to align with the weights in the refined decoupling Theorem~\ref{Grefdec} when we make the key estimate \eqref{eq:key}.

\vspace{1em}
\noindent
\textbf{Stopping Time Algorithm.} Let $\delta = \varepsilon^2$. We define the stopping time $\mathbf{k}(x)$ by the following rule:

\begin{itemize}
    \item If there exists $1 \le k < N$ such that
    \[
    \mathcal{A}_{k-1} f(x) > R^\delta  \mathcal{A}_k f(x),
    \]
    then define $\mathbf{k}(x)$ to be the smallest such $k$.
    
    \item If no such $k < N$ exists, set $\mathbf{k}(x) := N$.
\end{itemize}

\vspace{1em}
\noindent
\textbf{Stopping Regions.} Define measurable regions where the algorithm halts at each step:
\begin{align*}
\Omega_1 &:= \left\{ 
  x \in \mathbb{R}^3 : 
  \mathcal{A}_0 f(x) > R^\delta \mathcal{A}_1 f(x) 
\right\}, \\
\Omega_k &:= \left\{
  \begin{aligned}
    x \in \mathbb{R}^3 : \quad 
    &\mathcal{A}_j f(x) \le R^\delta \mathcal{A}_{j+1} f(x) 
      &&\text{for all } 0 \le j < k-1, \\
    &\mathcal{A}_{k-1} f(x) > R^\delta \mathcal{A}_k f(x)
  \end{aligned}
\right\}
&&\text{for } 2 \le k < N, \\
\Omega_N &:= \left\{
  \begin{aligned}
    x \in \mathbb{R}^3 : \quad 
    \mathcal{A}_j f(x) \le R^\delta \mathcal{A}_{j+1} f(x)
    &&\text{for all } 0 \le j < N
  \end{aligned}
\right\}.
\end{align*}

\vspace{1em}
\noindent
This procedure partitions $\mathbb{R}^3$ into disjoint stopping regions $\{\Omega_k\}_{1 \le k \le N}$, with each $\Omega_k$ corresponding to points where the function $f$ goes from being ``crowded'' to ``sparse'' in transition from scales $\rho_{k-1}$ to $\rho_{k}$ for the first time. Motivated by the analogy \(\mathcal{A}_k f \approx \mathbb{E}_{\rho_k} f\), we should think of each region \(\Omega_k\) as being approximately a union of \(\rho_{k-1}\)-cubes (see Proposition~\ref{eq:loc-const} for the precise formulation used in the proof).

Since there are $N\lesssim \e^{-1}$ regions $\Omega_k$, by pigeonholing, we may select $k\in[N]$ such that
\[ \int_X f\lesssim \e^{-1} \int_{\Omega_k\cap X} f.  \]

It follows by the definition of $\Omega_k$ that 
\begin{equation}\label{eq:time}
    \mu|\mathcal P_\mu| \lesssim_\e R^{(k-1)\delta}\int_{\Omega_k\cap X}\mathcal A_{k-1}f.
\end{equation}
Note that $R^{k\delta} \le R^{C\e}\approx 1$ for all $k\in[N]$. The rest of the analysis splits into two cases. The first case is that $1\le k<N$, which we treat presently. 
Define $\mathcal H_kf = \mathcal A_{k-1}f-\mathcal A_k f$ and note that for $x\in \Omega_k$, 
\[ \mathcal A_{k-1}f(x)\sim \mathcal H_kf(x)=|\mathcal H_kf(x)|\ge R^{\d} \mathcal A_kf(x), \]
so 
\[
\mu|\mathcal P_\mu| \lesssim_\e R^{(k-1)\delta}\int_{\Omega_k \cap X}|\mathcal H_kf|.
\]
The Fourier support of $\mathcal H_kf$ is contained in the annulus $\{\rho_{k}^{-1}<|\xi|<\rho_{k-1}^{-1}\}$. Consider a smooth partition of unity $\sum_\lambda \s_\lambda=1$ into dyadic annuli $\{|\xi|\sim\lambda\}$ in the range between $\rho_{k}^{-1}$ and $\rho_{k-1}^{-1}$ (where the $\s_\lambda$ are dilations of a single function). Since $\mathcal H_kf = \sum_\lambda \mathcal H_kf \ast \widecheck {\psi_\lambda}$, after dyadically pigeonholing, we may assume that the function $\mathcal H_kf$ is Fourier supported in the $\lambda$-dilation of $\mc{N}_{R^{-1}/\lambda^2}(\Gamma)$ (technically, we have to replace $\mathcal H_kf$ by $\mathcal H_kf*\widecheck{\s_\lambda}$, which by abuse of notation we will continue to denote by $\mathcal H_kf$).

Each $\theta$ has a corresponding angular sector of size $R^{-1/2}$. Write $\tau_k$ for angular sectors of size $R^{-1/2}/\lambda$, and $\theta\subset\tau_k$ if the sector corresponding to $\theta$ is contained in $\tau_k$, so that
\[
f = \sum_{\tau_k}\sum_{\theta\subset\tau_k} f_\theta =: \sum_{\tau_k}f_{\tau_k},
\]
and let
\[
\mathcal A_{j}f_{\tau_k} := \sum_{\theta\subset \tau_k}f_\theta * \tilde\w_{0,\theta}*\tilde{\w}_{1,\theta}*\cdots*\tilde{\w}_{j,\theta}.
\]
(We will only need $\mathcal A_{k-1}f_{\tau_k},\mathcal A_kf_{\tau_k}.)$

Likewise, define
\[
\mathcal H_k f_{\tau_k} := \mathcal A_{k-1}f_{\tau_k}-\mathcal A_{k}f_{\tau_k},
\]
so that
\[\mathcal H_kf=\sum_{\tau_k}\mathcal H_kf_{\tau_k}. \]
Note that $\mathcal H_k f_{\tau_k}$ is Fourier supported in a $\lambda\times R^{-1/2}\times R^{-1}\lambda^{-1}$-cone plank contained in $\{|\xi|\sim \lambda\}$, lying above the angular sector $\tau_k$.

By incorporating the frequency localization of \(\mathcal{H}_k f\), we are now positioned to apply the refined cone decoupling theorem (Theorem~\ref{Grefdec}) in order to further estimate \(\mu |\mathcal{P}_\mu|\). More precisely, we will apply Theorem~\ref{Grefdec} to a level set \(U_{\alpha, \beta} \subset \Omega_k\), where \(|\mathcal{H}_k f| \sim \alpha\) and 
\[
\sum_{\tau_k} |\mathcal{H}_k f_{\tau_k}|^2 * W_{M, \tau_k} \sim \beta.
\]
Here, the weight functions \(W_{M, \tau_k}(x) = w_{M, \tau_k}(\lambda x)\), with \(w_{M, \tau_k}\) given by Theorem~\ref{Grefdec}.

By the Cauchy--Schwarz inequality,
\[
|\mathcal H_kf| \le \sum_{\tau_k}|\mathcal H_kf_{\tau_k}| \le (\#\tau_k)^{1/2}(\sum_{\tau_k}|\mathcal H_kf_{\tau_k}|^2)^{1/2}\le R^{1/4}(\sum_{\tau_k}|\mathcal H_kf_{\tau_k}|^2)^{1/2}.
\]
Furthermore, note that $|\mathcal H_kf_{\tau_k}|^2$ is Fourier supported in an origin-centered $\sim\lambda\times R^{-1/2}\times \lambda^{-1}R^{-1}$-cone plank, so we have $|\mathcal H_kf_{\tau_k}|^2\lesssim |\mathcal H_k f_{\tau_k}|^2\ast W_{M,\tau_k}$ pointwise uniformly. Hence, if $\sum_{\tau_k}|\mathcal H_kf_{\tau_k}|^2* W_{M,\tau_k}\le R^{-100}$, then $|\mathcal H_kf|\le R^{-90}$. By a round of dyadic pigeonholing, we obtain for some $\alpha,\beta > R^{-100}$,
\begin{equation}\label{eq:stopping-time-level-set}
\mu|\mathcal P_\mu|\lessapprox \alpha |U_{\alpha,\beta}\cap \Omega_k\cap X| + R^{-90},
\end{equation}
in which
\[
U_{\alpha,\beta} = \{x\in\R^3:|\mathcal H_kf|\sim \alpha, \sum_{\tau_k}|\mathcal H_kf_{\tau_k}|^2*{W}_{M
,\tau_k} \sim \beta\}.
\]
   Before applying refined decoupling to bound $|U_{\a,\b}|$, we exploit the heuristic that $|\mathcal H_kf|$ and $\sum_{\tau_k}|\mathcal H_kf_{\tau_k}|^2*W_{M,\tau_k}$ are locally constant on $\rho_{k-1}$-balls to study the structure of $U_{\a,\b}\cap\Omega_k\cap X$. 

\begin{prop}[Locally constant property]\label{prop:loc-const}
    For each $\varepsilon > 0$, there exists a constant $K = K_\varepsilon > 1$ such that the following holds. Suppose $x \in U_{\alpha,\beta} \cap \Omega_k$. Then for any $\rho_{k-1}$-cube $Q$ containing $x$, we have
\[
Q \subset \tilde{U}_{\alpha,\beta} \cap \tilde{\Omega}_k,
\]
where
\[
\tilde{U}_{\alpha,\beta} \cap \tilde{\Omega}_k = \left\{
\begin{aligned}
    &|\mathcal{H}_k f| \sim_K \alpha, \quad
    \sum_{\tau_k} |\mathcal{H}_k f_{\tau_k}|^2 * W_{M,\tau_k} \sim \beta, \\
    &\mathcal{A}_{k-1} f \sim_K \alpha, \quad
    \mathcal{A}_k f \lesssim_K R^{-\delta} \alpha
\end{aligned}
\right\},
\]
and $a \sim_K b$ means $K^{-1} b \le a \le K b$.

In particular, for every $\rho_{k-1}$-cube $Q$ and every unit ball $q$ with $q \cap Q \ne \emptyset$, we have
\begin{equation} \label{eq:loc-const}
\frac{|\tilde{U}_{\alpha,\beta} \cap \tilde{\Omega}_k \cap Q|}{|Q|} \gtrsim |U_{\alpha,\beta} \cap \Omega_k \cap q|.
\end{equation}

\end{prop}
\begin{proof}[Proof of Proposition \ref{prop:loc-const}; the locally constant property]
    We begin by noting that for any $\theta$, any nonnegative function $g$, and any $y,z,z'\in\R^3$ such that $|z-z'|\lesssim\rho_{k-1}$, we have
    \[
    g(y)\tilde{\omega}_{k-1,\theta}(z-y) \sim g(y)\tilde{\omega}_{k-1,\theta}(z'-y). 
    \]
    It follows by nonnegativity of $f$ that $\mathcal A_{k-1}f(z)\sim \mathcal A_{k-1}f(z')$ for any such $z,z'$. Similar reasoning applied to $\tilde\omega_{k,\theta}$ shows that $\mathcal A_{k}f(z)\sim \mathcal A_{k}f(z')$ if $|z-z'|<\rho_{k-1}$ (actually, this holds as long as $z,z'$ belong to the same translation of a $\rho_k\times R^{1/2}\times R$-stack of lightplanks parallel to $\theta^*$, but we have no use for this stronger fact). For the statement about $\sum_{\tau_k}|\mathcal H_kf_{\tau_k}|^2*W_{M,\tau_k}$, we note that $W_{M,\tau_k}$ is approximately constant (depending only on $M=M(\e)$) on planks with dimensions $\lambda^{-1}\times R^{1/2}\times \lambda R$, which contain $\rho_{k-1}$-balls since $\lambda^{-1}\ge \rho_{k-1}$ and $\lambda R\ge \rho^{-1} R\ge R^{1/2}\ge \rho_{k-1}$. 
    
    Since $x\in\Omega_k$, for any $y\in Q$, we have $\mathcal A_{k-1}f(y)\sim \mathcal A_{k-1}f(x)\sim |\mathcal H_k f(x)|$, and $\mathcal A_{k}f(y) \sim \mathcal A_{k}f(x)< R^{-\delta}\mathcal A_{k-1}f(x)$. In particular, for any $y\in Q$, $\mathcal A_{k}f(y) < \frac{1}{2}\mathcal A_{k-1}f(y)$, since we may assume from the onset that $R\ge K^{\e^{-2}}$. Hence, for any $y\in Q$, by the reverse triangle inequality and our assumption that $x\in U_{\alpha,\beta}$,
    \[
    |\mathcal H_kf(y)| \ge \mathcal A_{k-1}f(y) - \mathcal A_{k}f(y) \gtrsim\alpha.
    \]
    This finishes the proof of the containment $Q\subset \tilde U_{\alpha,\beta}\cap\tilde\Omega_k$. The proof of \eqref{eq:loc-const} follows immediately.
\end{proof}

By refined decoupling from Theorem~\ref{Grefdec}, we have
\begin{equation}\label{refdecresult}
\alpha^6|\tilde{U}_{\alpha,\beta}\cap \tilde{\Omega}_k| \lesssim \int_{\tilde U_{\a,\b}\cap\tilde\Omega_k}|\mathcal H_kf|^6\lesssim_\e R^\e  \b^2\sum_{\tau_k}\int_{\R^3}|\mathcal H_kf_{\tau_k}|^2. 
\end{equation}
Note that $\b\sim \sum_{\tau_k}|\mathcal H_kf_{\tau_k}|^2*W_{M,\tau_k}$ on $\tilde{U}_{\a,\b}\cap\tilde{\Omega}_k$, so $\b\sim \|\sum_{\tau_k}|\mathcal H_kf_{\tau_k}|^2*W_{M,\tau_k}\|_{L^\infty(\tilde{U}_{\a,\b}\cap\tilde{\Omega}_k)}$.

We will first bound the $L^2$ integral on the right hand side of \eqref{refdecresult}. Note that the weight functions are $L^1$-normalized, meaning that $\tilde{\omega}_{j,\theta}$ satisfy $\max_{j,\theta}\int\tilde{\omega}_{j,\theta}\sim 1$. Since $k<N\lesssim\e^{-1}$, and $f_\theta\ge 0$, we have
\begin{equation}
    \|\mathcal H_kf_{\tau_k}\|_{L^1(\R^3)}\le\|\mathcal A_{k-1}f_{\tau_k}\|_{L^1(\R^3)}+\|\mathcal A_kf_{\tau_k}\|_{L^1(\R^3)} \lesssim \sum_{\theta\subset\tau_k}\int_{\R^3} f_\theta.
\end{equation}
Therefore, since each $f_\theta\lesssim 1$, and $\|\mathcal H_kf_{\tau_k}\|_{L^\infty(\R^3)}\lesssim \#(\theta\subset\tau_k)$, we have
\begin{align*}
    \sum_{\tau_k}\int_{\R^3} |\mathcal H_kf_{\tau_k}|^2 &\lesssim \max_{\tau_k}\|\mathcal H_kf_{\tau_k}\|_{L^\infty(\R^3)}\cdot\sum_{\tau_k}\big(\|\mathcal A_{k-1}f_{\tau_k}\|_{L^1(\R^3)}+\|\mathcal A_{k}f_{\tau_k}\|_{L^1(\R^3)}\big)\\
    &\lesssim \max_{\tau_k}\#(\theta\subset\tau_k)\sum_{\tau_k}(\sum_{\theta\subset\tau_k}\int_{\R^3} f_\theta) \\
    &= \max_{\tau_k}\#(\theta\subset\tau_k)\|f\|_{L^1(\R^3)} \sim \max_{\tau_k}\#(\theta\subset\tau_k)R^{3/2}|\mathcal P_\mu|.
\end{align*}
Hence we have the bound
\begin{align}\label{eq:ftau-bounded}
    \alpha^6|\tilde U_{\alpha,\beta}\cap\tilde \Omega_k|&\lessapprox  \max_{\tau_k}\#(\theta\subset\tau_k)^3\|\sum_{\tau_k}|\mathcal H_kf_{\tau_k}|*W_{M,\tau_k}\|_{L^\infty(\tilde{U}_{\a,\b}\cap\tilde{\Omega}_k)}^2 R^{3/2}|\mathcal P_\mu|.
\end{align}
Next, we apply the triangle inequality and nonnegativity of the functions $f_\theta,\tilde{\omega}_{j,\theta}$ to say that 
\begin{align}
\sum_{\tau_k}|\mathcal H_kf_{\tau_k}|*W_{M,\tau_k}    &\le \sum_{\tau_k}\mathcal A_{k-1}f_{\tau_k}*W_{M,\tau_k}+\sum_{\tau_k}\mathcal A_kf_{\tau_k}*W_{M,\tau_k} \label{pick}.
\end{align}
The function \(\sum_{\tau_k} \mathcal{A}_k f_{\tau_k} * W_{M, \tau_k}\) is pointwise comparable to \(\mathcal{A}_k f\), since \(\tilde{\omega}_{k, \theta} * W_{M, \tau_k} \sim \tilde{\omega}_{k, \theta}\) for \(M > \varepsilon^{-1}\) and \(\theta \subset \tau_k\). This holds because \(\tilde{\omega}_{k, \theta}\) is localized to an origin-centered \(\rho_k \times R^{1/2} \times R\) box, while \(W_{M, \tau_k}\) is localized to an origin-centered \(\lambda^{-1} \times R^{1/2} \times \lambda R\) \emph{axis-parallel} subset of that box.

\medskip

\textbf{The key point.} Since \(\rho_k^{-1} < \lambda < \rho_{k-1}^{-1}\) and \(\rho_k = R^{\varepsilon} \rho_{k-1}\), it follows that for \(M > \varepsilon^{-1}\) and \(\theta \subset \tau_k\),
\begin{equation} \label{eq:key}
\tilde{\omega}_{k-1, \theta} * W_{M, \tau_k} \lesssim R^{\varepsilon} \, \tilde{\omega}_{k-1, \theta} * \tilde{\omega}_{k, \theta}.
\end{equation}

Indeed, the left-hand side is approximately \(\lambda / R^{3/2}\) on a \(\lambda^{-1} \times R^{1/2} \times R\) block, while the right-hand side (before inserting the \(R^{\varepsilon}\) factor) is approximately \(\rho_k^{-1} / R^{3/2}\) on a larger \(\rho_k \times R^{1/2} \times R\) block. Since \(\lambda < R^{\varepsilon} \rho_k^{-1}\), the left-hand side may exceed $\tilde\w_{k-1,\theta}*\tilde\w_{k,\theta}$ by at most a factor of \(R^{\varepsilon}\), which justifies the inequality.

Therefore, picking up from \eqref{pick}, we have
\[ \sum_{\tau_k}|\mathcal H_kf_{\tau_k}|*W_{M,\tau_k}\lesssim R^\e\,\mathcal A_kf. \]
Now, by the definition of $\tilde{U}_{\a,\b}\cap\tilde{\Omega}_k$,
\[ \|\sum_{\tau_k}|\mathcal H_kf_{\tau_k}|*W_{M,\tau_k}\|_{L^\infty(\tilde{U}_{\a,\b}\cap\tilde{\Omega}_k)}\lesssim R^\e \|\mathcal A_kf\|_{L^\infty(\tilde{U}_{\a,\b}\cap\tilde{\Omega}_k)}\lesssim R^{\e}R^{-\delta} \a. \]

The number of $\theta\subset\tau_k$ is at most $\sim \lambda^{-1}$, so \eqref{eq:ftau-bounded} is bounded by
\begin{align*}
    &\lessapprox \lambda^{-3}\a^2 |\mc{P}_\mu|R^{3/2}
\end{align*}
and hence
\begin{equation}\label{eq:decoupling-input}
    \alpha^4|\tilde{U}_{\alpha,\beta}\cap \tilde{\Omega}_k| \lessapprox \lambda^{-3}\|f\|_{L^1(\R^3)}.
\end{equation}

Summing over dyadic $\rho_{k-1}$-cubes, we have by the locally constant property Proposition \ref{prop:loc-const} (and the assumption that $k < N$, so there are at most $O(1)$-many unit balls $q\subset X$ in any $\rho_{k-1}$-cube),
\begin{align}\label{eq:loc-const-prop}
|\tilde{U}_{{\alpha,\beta}}\cap \tilde{\Omega}_k| \ge |Q|\sum_{\substack{Q\ \text{dyadic}\ \rho_{k-1}\text{-cube}\\Q\cap X\ne \emptyset}}\frac{|\tilde U_{{\alpha,\beta}}\cap \tilde\Omega_k\cap Q|}{|Q|}\gtrsim \rho_{k-1}^3|U_{{\alpha,\beta}}\cap \Omega_k\cap X|. 
\end{align}
Plugging \eqref{eq:loc-const-prop} into \eqref{eq:decoupling-input} and using $\lambda \ge \rho_{k}^{-1}=R^{-\e}\rho_{k-1}^{-1}$, we get
\begin{equation}\label{eq:LC+Dec}
\alpha^4|U_{\alpha,\beta}\cap \Omega_k\cap X| \lessapprox \|f\|_{L^1(\R^3)}.
\end{equation}
Plugging \eqref{eq:stopping-time-level-set} into \eqref{eq:LC+Dec}, we get
\[
\bigg(\frac{\mu|\mathcal P_\mu|}{|U_{\alpha,\beta}\cap\Omega_k\cap X|}\bigg)^4|U_{\alpha,\beta}\cap\Omega_k\cap X|\lessapprox \|f\|_{L^1(\R^3)},
\]
so
\[
\mu^4|\mathcal P_\mu|^4 \lessapprox |X|^3\|f\|_{L^1(\R^3)} \sim |X|^3\cdot R^{3/2}|\mathcal P_\mu|.
\]
Since we assume $|X|\sim (R/\rho)^3$, where $1\le \rho\le R^{1/2}$, we obtain the desired bound
\[
\mu^{4/3}|\mathcal P_\mu| \lessapprox |X|^{4/3}.
\]

It remains to treat the case \( k = N \), where the stopping time algorithm does not terminate before reaching the top scale. Since \( \mathbf{k}(x) = N \) means that the inequality
\[
\mathcal{A}_j f(x) \le R^\delta \mathcal{A}_{j+1} f(x)
\]
holds for all \( 0 \le j < N \), we may iterate this bound and apply linearity and $\delta=\e^2$ to deduce
\begin{equation}\label{case:k=N}
\mu|\mathcal{P}_\mu| \lesssim_\varepsilon R^{N\delta} \int_{\Omega_N \cap X} \mathcal{A}_{N} f \lessapprox \sum_\theta \int_X \mathcal A_Nf_\theta.
\end{equation}
We introduce one last piece of notation. For each $\theta$, and $P\in \mathcal P_{\mu,\theta}$, let
\[
\mathcal A_{N,\theta}\phi_P := \phi_P\ast \tilde\w_{0,\theta}*\tilde\omega_{1,\theta}*\dots*\tilde\omega_{N,\theta}.
\]
Note that $\mathcal A_{N,\theta}\phi_P(x) \sim \frac{|P\cap B_\rho|}{|B_\rho|}=\rho^{-1}$ for $x$ in a concentric plank $\tilde P$ of approximate dimensions $\rho\times R^{1/2}\times  R$ containing $P$, and that $\mathcal A_{N,\theta}\phi_P(x) \le R^{-50}$ if $x\notin C_\e R^\e\tilde P$ for $C_\e>1$ a large enough constant.

By the $\rho$-separation of the set $X$, we have for each fixed plank $P \in \mathcal{P}_{\mu,\theta}$,
\begin{equation}\label{eq:single-fat-tube}
    \int_{X} \mathcal{A}_{N,\theta} \phi_P 
    \lessapprox \rho^{-1} |X \cap R^\varepsilon \tilde{P}| +  R^{-50}
    \lessapprox \rho^{-1}\frac{R^{3/2}}{\rho^2} + R^{-50}.
\end{equation}
Summing inequality \eqref{eq:single-fat-tube} over all planks $P \in \mathcal{P}_{\mu,\theta}$ and all $\theta$, and using \eqref{case:k=N}, we obtain
\[
\mu \lessapprox \frac{R^{3/2}}{\rho^3}.
\]

Since the family $\mathcal{P}_\mu$ consists of $K$-incomparable planks, we have the upper bound $|\mathcal{P}_\mu| \lesssim_K R^2$. Combining this trivial upper bound for $|\mathcal{P}_\mu|$ with the above estimate for $\mu$, and recalling that $|X|\sim(R/\rho)^3$, we conclude that
\[
\mu^{4/3}|\mathcal{P}_\mu| 
\lessapprox \left(\frac{R^{3/2}}{\rho^3}\right)^{4/3} R^2 
= \frac{R^4}{\rho^4} \sim |X|^{4/3}.
\]

This completes the analysis of the case $k = N$ and finishes the proof of Theorem~\ref{thm:main}.
\end{proof}

\begin{remark}
    By rescaling, Theorem \ref{thm:main} implies Theorem \ref{thm:rect1} stated in the introduction. To see this, assume Theorem \ref{thm:main}, and let $X\subset[0,1]^2\times[1,2]$ be well-spaced, and let $\tau\in[|X|^{-1/3},|X|^{-1/6}]$, and $\mathcal P_{\mu,\tau}$ be a collection of pairwise incomparable $\mu$-rich $\tau\times\tau^2$ tangency rectangles. By point-lightplank duality, $\mathcal P_{\mu,\tau}$ can be identified with a collection $\mathcal P$ of pairwise incomparable $\mu$-rich (in $X$) $1\times\tau\times\tau^2$-lightplanks. Rescaling by $R:= \tau^{-2}$ yields a set $\tilde X\subset[0,R]^2\times[R,2R]$ which is maximally $\rho:=|X|^{-1/3}\tau^{-2}\in[1,R^{1/2}]$-separated, and a set $\tilde{\mathcal P}$ of $\mu$-rich $1\times R^{1/2}\times R$ lightplanks.

    If $\tau\in(|X|^{-1/3},1)$, then the estimate of Theorem~\ref{thm:rect1} follows immediately using the well-spaced condition, since every $\tau\times\tau^2$ rectangle is automatically $\sim|X|\tau^3$-rich, and there are $\sim \tau^{-4}$-many such $\tau\times\tau^2$ rectangles.
\end{remark}

We will use Theorem~\ref{thm:rect1} to break the 3/2-barrier of Problem~\ref{prob:tangency-sites} 
for well-spaced sets of circles.  
The following is Theorem~\ref{thm:well-spaced-tang}, restated for convenience.

\begin{thm}
    For every $\varepsilon>0$, there is a constant $A_\e$ so that the following holds.  
Let $X\subset [0,1]^2\times[1,2]$ be well-spaced, and let $\mathcal C_X$ denote the collection of circles
\[
\{C_{z,r}:(z,r)\in X\},
\]
where $C_{z,r}$ denotes the circle with center $z\in\R^2$ and radius $r>0$.

If no three circles of $\mathcal C_X$ are tangent at a point, then the following estimate holds:
\[
|\mathcal T_{pair}(\mathcal C_X)|\le  A_\e|X|^{25/18+\varepsilon}.
\]
\end{thm}

\begin{proof}
    We will actually show that for an arbitrary well-spaced set of circles $X\subset[0,1]^2\times[1,2]$, the set
    \[
    \mathcal T(\mathcal C_X) = \{z\in\R^2 : \text{at least two circles of $\mathcal C$ are internally tangent at $z$}\}
    \]
    has cardinality $O_\e(|X|^{25/18+\e})$. By the discussion after the statement of Problem \ref{prob:tangency-sites}, this implies the corresponding estimate for $|\mathcal T_{pair}(\mathcal C_X)|$ when no three circles of $\mathcal C_X$ are tangent at a point.
    
    Begin by letting $\mathcal P$ be a maximal pairwise incomparable collection of $2$-rich (in $X$) $1\times |X|^{-1/3}\times |X|^{-2/3}$ lightplanks. If $q\in\mathcal T(\mathcal C_X)$ is given, then there necessarily exists a lightplank $P\in\mathcal P$ and two points $(x,r),(x',r')\in X\cap P$ such that the circles $C_{x,r},C_{x',r'}$ are internally tangent at $q$. Therefore, we have the following union bound for $|\mathcal T(\mathcal C_X)|$:
    \begin{equation}\label{eq:union-bound}
    |\mathcal T(\mathcal C_X)| \le \sum_{P\in\mathcal P}|\mathcal T(\mathcal C_{X\cap P})|.
    \end{equation}
    By dyadically pigeonholing, we may assume that every $P$ in the sum in \eqref{eq:union-bound} satisfies $|X\cap P|\sim \mu$ for some $\mu \ge 1$. Denote the collection of such planks as $\mathcal P_{\mu}$, so that
    \begin{equation}\label{eq:pigeonhole-planks}
    |\mathcal T(\mathcal C_X)| \lessapprox  \sum_{P\in\mathcal P_{\mu}}|\mathcal T(\mathcal C_{X\cap P})|
    \end{equation}
    Applying Wolff's bound $|\mathcal T(\mathcal C_{X\cap P})| \lessapprox |X\cap P|^{3/2} \sim \mu^{3/2}$ for each $P\in\mathcal P_\mu$, we obtain
    \[
    |\mathcal T(\mathcal C_X)|\lessapprox \mu^{3/2}|\mathcal P_\mu|.
    \]
    Applying Theorem \ref{thm:rect1}, we obtain
    \[
    |\mathcal T(\mathcal C_X)|\lessapprox \mu^{1/6}|X|^{4/3}.
    \]
    Since $X$ is well-spaced, it is $\sim |X|^{-1/3}$-separated, so $\mu\lesssim |X|^{1/3}$. Plugging this into the preceding expression yields
    \[
    |\mathcal T(\mathcal C_X)|\lessapprox |X|^{4/3+1/18} = |X|^{25/18},
    \]
    as claimed.
\end{proof}

\section{Decoupling}\label{sec:decoupling}
There is a natural resemblance between the sequence \(\mathcal{A}_k f\) and its increments \(\mathcal{H}_k f := \mathcal{A}_{k-1}f - \mathcal{A}_k f\) appearing in the proof of Theorem \ref{thm:main}, and the classical setting of martingales and their quadratic variations. In probability theory, the Burkholder–Davis–Gundy inequalities compare the size of a martingale to the square function formed from its increments \cite{burkholder1972integral}. Our setup, though entirely deterministic, echoes this structure: \(\mathcal{A}_k f\) behaves like a smoothed or “filtered” version of \(f\), while the increments \(\mathcal{H}_k f\) measure the energy transferred between scales. This analogy is purely heuristic, but we hope it helps place the arguments of Section~\ref{sec:rectangles} in a natural light.

Refined high–low frequency analysis, particularly the interplay between a function \(f\) and its square function \(\sum_\theta |f_\theta|^2\), emerged in earlier works such as~\cites{du2018pointwise, guth2020falconer}, and featured implicitly in the improved decoupling theorem for the parabola by Guth--Maldague--Wang~\cite{guth2021improved}, though the structural framework was not explicitly isolated there. A related approach, with a clearer articulation of the relationship between \(f\) and its square function, appears in the work of Fu--Guth--Maldague~\cite{fu2023decoupling} on decoupling inequalities for short Dirichlet sequences.

\subsection{Proof of the refined decoupling Theorem \ref{Grefdec}}
Let
\[
\Gamma = \{(\xi_1,\xi_2,\xi_3)\in\R^3:\xi_3 = \sqrt{\xi_1^2 + \xi_2^2}, 1 \le \xi_3 \le 2\}
\]
be the usual segment of the cone. For $\{a_\theta\}$ a partition of $S^1$ into $R^{-1/2}$-arcs (assuming $R$ is of the form $2^{2^{k_0}}$), define the set of cone planks $\Theta_\Gamma(R)=\{\theta\}$ to be 
\[ \theta=\mc{N}_{R^{-1}}(\Gamma)\cap \{\arg(\xi_1,\xi_2)\in a_\theta\} . \]
Let $w_{M,\theta}=|\det A_\theta|w_M\circ A_\theta$ be weight functions, in which $A_\theta$ is a linear function mapping the origin-centered $1\times R^{1/2}\times R$-cone plank $\theta^*$ to the unit cube, and where $w_M(x)=\frac{1}{(1+|x|)^M}$. 

\begin{thm}[Refined decoupling for $\Gamma$] \label{Grefdec} For each $\e>0$, there exists $C_\e\in(0,\infty)$ so that the following holds. Let $R\ge 1$ and consider $\mc{N}_{R^{-1}}(\Gamma)$. For each $M\ge C_\e$ and each Schwartz function $g_\theta\colon \R^3\to \C$ with Fourier support in $\theta$ we have
\begin{equation*}
    \int_{U_\b\cap B_R}|\sum_\theta g_\theta|^6 \le C_\e R^\e \b^2 \sum_\theta \int_{\R^3}|g_\theta|^2
\end{equation*}    
for each $R$-ball $B_R\subset\R^3$, where
\[ U_{\b}:=\{x\in \R^3: \b\le \sum_\theta|g_\theta(x)|^2*w_{M,\theta}\le 2\b\} .\] 
\end{thm}

We prove Theorem \ref{Grefdec} in two parts, first by establishing Theorem \ref{Prefdec}, an analogous refined decoupling theorem for the parabola, and then using a lifting technique based on cylindrical decoupling to deduce Theorem \ref{Grefdec} for the cone from Theorem \ref{Prefdec}. 

\subsection{Refined decoupling for the parabola (Proof of Theorem \ref{Grefdec} part 1/2)}

Let $\mc{N}_{R^{-1}}({Par})=\mc{N}_{R^{-1}}(\{(t,t^2):-1\le t\le 1\})$. Let $\theta =\theta(l)$ be the intersection of $\left([lR^{-1/2},(l+1)R^{-1/2})\times\R\right)\cap\mc{N}_{R^{-1}}(Par)$ and the end pieces, so $\mc{N}_{R^{-1}}(Par)=\bigsqcup\theta$. 

Let $w_M$ be the weight function $w_M\colon \R^2\to[0,\infty)$ defined by $w_M(x)=\frac{1}{(1+|x|)^{M}}$. For each $\theta=\theta_l$, let $A_\theta\colon\R^2\to \R^2$ be the linear transformation \[A_\theta\colon(x_1,x_2)\mapsto x_1R^{1/2}(1,2lR^{-1/2})+x_2R(-2lR^{-1/2},1).\] Let $w_{M,\theta}\colon\R^2\to[0,\infty)$ denote $w_M\circ A_\theta^{-1}(x)$. 

\begin{thm}[Refined decoupling for ${Par}$] \label{Prefdec} For each $\e>0$ and $M>0$ sufficiently large depending on $\e$, there exists $C_{\e,M}\in(0,\infty)$ so that the following holds. Let $R\ge 1$ and consider $\mc{N}_{R^{-1}}(Par)$. For each Schwartz function $g_\theta\colon \R^2\to \C$ with Fourier support in $\theta$ we have
\[ \int_{U_{\b}\cap B_R}|\sum_\theta g_\theta|^6\le C_{\e,M} R^\e \b^2 \sum_\theta \int_{\R^2}|g_\theta|^2   \]
for each $R$-ball $B_R\subset\R^2$, where
\[ U_{\b}:=\{x\in \R^2: \b\le \sum_\theta|g_\theta|^2*w_{M,\theta}(x)\le 2\b\} .\] 
\end{thm}

We prove Theorem \ref{Prefdec} by defining $D(M,\e,R)$ to be the smallest constant such that 
\[ \int_{U_{\b}\cap B_R}|\sum_\theta g_\theta|^6 \le  D(M,\e,R) \b^2 \sum_\theta \int_{\R^2}|g_\theta|^2\]
for each $g_\theta$ and $U_{\b}$ from the set-up of Theorem \ref{Prefdec}. We will first perform canonical decoupling into $\tau$ which are parabola blocks in $\mc{N}_{R^{1/8}}(Par)$. Each $\tau$ has approximate dimensions $R^{-1/8}\times R^{-1/4}$. 

\begin{prop}[Multiscale inequality]\label{multi}For each $\e>0$, $M\in\N^+$, and $R\ge 1$, 
\[ D(M,\e,R)\lesssim_\d R^\d D(M,\e,R^{3/4})+\emph{RapDec}(1). \]
    
\end{prop}

\begin{proof} Let $\p_{R^{-1/4}}$ be a standard bump function adapted to $B^{(2)}(0,R^{-1/4})$. Cover the set $U_{\b}\cap B_R$ by a finitely overlapping set $b$ of $R^{1/4}$-balls. Let $\p_b$ be $|\widecheck{\p}_{R^{-1/4}}|^2(x-c_b)$, where $c_b$ is the center of $b$.

    Begin by performing some pigeonholing steps. For each $b$ and each $R^{-1/8}$-sector $\tau$, find $\b_{\tau,b}>0$ dyadic values in the range \[ R^{-100}\max_\theta\|g_\theta\|_{L^\infty(\R^2)}\le \b_{\tau,b}^{1/2}\le R^{1/2}\max_\theta\|g_\theta\|_{L^\infty(\R^2)} \] 
    such that
    \[ \int_{R^\e b}|g_\tau|^6\lesssim (\log R)^6\int_{U_{\b_{\tau,b}}\cap R^\e b}|g_\tau|^6+C R^{-100}\max_\theta\|g_\theta\|_{L^\infty(\R^2)}^6,\]
    where $R^\e b$ is the dilation of $b$ by a factor $R^\e$ with respect to the center of $b$ and $U_{\b_{\tau,b}}=\{x\in\R^2:  \sum_{\theta\subset\tau}|g_\theta|^2*w_{M,\theta}(x)\sim\b_{\tau,b} \}$. 

    The next step is pigeonholing the $\tau$ on each ball. Find $\mc{T}(b)$ so that
    \[ \int_{U_{\b}\cap b}|g|^6\lesssim (\log R)^6 \int_{U_{\b}\cap b}|\sum_{\tau\in\mc{T}(b)}g_\tau|^6+R^{-100}\max_\theta\|g_\theta\|_{L^\infty(\R^2)}^6,\]
    where $\b_{\tau,b}$ is within a factor of $2$ of a single value  for all $\tau\in \mc{T}(b)$.

    Finally, pigeonhole the $b$ so that for some set $\mc{B}$, 
    \[ \sum_b \int_{U_{\b}}|\sum_{\tau\in\mc{T}(b)}g_\tau|^6\lesssim (\log R)\sum_{b\in\mc{B}}\int_{U_{\b}}|\sum_{\tau\in\mc{T}(b)}g_\tau|^6 \]
    and $\#\mc{T}(b)$ is within a factor of $2$ of a single value for all $b\in\mc{B}$.   

    Now apply canonical decoupling into the $\tau$:
    \begin{align*} 
    \sum_{b\in\mc{B}}\int_{U_{\b}}|\sum_{\tau\in\mc{T}(b)}g_\tau|^6&\lesssim_\d R^\d \sum_{b\in\mc{B}} \left(\sum_{\tau\in\mc{T}(b)}\left(\int_{\R^2}|g_\tau \p_b|^6\right)^{2/6}\right)^{6/2} .
    \end{align*}
By the rapid decay of $\p_b$, each integral is bounded by 
\[ \int_{R^\e b}|g_\tau|^6+C_\e R^{-1000}\max_\theta\|g_\theta\|_{L^\infty(\R^2)}^6. \]
The pigeonholing steps mean that $\#\mc{T}(b)$ is approximately constant for all $b\in\mc{B}$, and for some $\b_0$, $\b_0\sim\b_{\tau,b}$ for all $\tau\in \mc{T}(b)$ and all $b\in\mc{B}$. It therefore remains to bound
\[ \#\mc{T}(b)^2\sum_{b\in\mc{B}}\sum_{\tau\in\mc{T}(b)}\int_{U_{ \b_0}(\tau)\cap R^\e b}|g_\tau|^6, \]
 in which $U_{ \b_0}(\tau)=\{x\in\R^2:\sum_{\theta\subset\tau}|g_\theta|^2*w_{M,\theta}(x)\sim\b_0\}$. Bound the previous displayed expression by summing over all $\tau$ and combining the integrals over the $b$: 
\[ R^{3\e} \#\mc{T}(b)^2\sum_\tau\int_{U_{\b_0}(\tau)\cap B_R}|g_\tau|^6. \]
By parabolic rescaling, each integral is bounded by  
\[ \int_{U_{ \b_0}(\tau)\cap B_R}|g_\tau|^6\le D(N,\e,R^{3/4})\b_0^2\sum_{\theta\subset\tau}\int_{\R^2}|g_\theta|^2. \]
It remains to show that 
\[ \#\mc{T}(b)\b_0\lesssim \b.\]
This is true since for each $\tau\in\mc{T}(b)$, there is an $x_\tau\in R^\e b$ such that $\b_0\sim \sum_{\theta\subset\tau}|g_\theta|^2*w_{M,\theta}(x_\tau)$. Let $x\in U_{\b}\cap b$. Since $|x-x_\tau|<R^\e R^{1/4}$, 
\[ \sum_{\theta\subset\tau}|g_\theta|^2*w_{M,\theta}(x_\tau)\sim \sum_{\theta\subset\tau}|g_\theta|^2*w_{M,\theta}(x). \]
Finally, note that
\[ \#\mc{T}(b)\b_0\lesssim \#\mc{T}(b)\sum_{\theta\subset\tau}|g_\theta|^2*w_{M,\theta}(x)\lesssim\sum_\theta|g_\theta|^2*w_{M,\theta}(x)\sim\b. \]
    
We also have that $R^{-100}\|g_\theta\|_{L^\infty(\R^2)}^6\lesssim R^{-50}\b^2\sum_{\theta}\int_{\R^2}|g_\theta|^2$. 
(It suffices to assume that $\b>R^{-10}\max_\theta\|g_\theta\|_{L^\infty(\R^2)}^2$ from the beginning.)
\end{proof}

\begin{proof}[Proof of Theorem \ref{Prefdec}] Let $\e>0$. By the multiscale inequality Proposition \ref{multi} iterated $k$ times, we have
\[ D(N,\e,R)\le C_\d^k R^{k\d}D(N,\e,R^{(3/4)^k})+kR^{-100}. \]

Let $k$ be the smallest integer so $(3/4)^k\le \e<(3/4)^{k-1}$. Then $k\le 1+ \ln \e^{-1}/\ln(4/3)$. Choose $\d>0$ so that $k\d<\e$.  Since $D(N,\e,R^{\e})\lesssim R^{C\e}$, it follows that
\[ D(N,\e,R)\lesssim C_\e R^{(C+1)\e}, \]
as desired. 
\end{proof}

\subsection{Refined decoupling for the cone (Proof of Theorem \ref{Grefdec} part 2/2)}

Now we prove Theorem \ref{Grefdec}, which we restate here for convenience.

\begin{thm}[Refined decoupling for $\Gamma$]\label{thm:proof-cone-dec} For each $\e>0$ and $M>0$ sufficiently large depending on $\e$, there exists $C_{\e,M}\in(0,\infty)$ so that the following holds. Let $R\ge 1$ and consider $\mc{N}_{R^{-1}}(\Gamma)$. For each Schwartz function $g_\theta\colon \R^3\to \C$ with Fourier support in $\theta$ we have
\begin{equation}
    \label{mainineq} \int_{U_\b\cap B_R}|\sum_\theta g_\theta|^6 \le C_{\e,M} R^\e \b^2 \sum_\theta \int_{\R^3}|g_\theta|^2
\end{equation}    
for each $R$-ball $B_R\subset\R^3$, where
\[ U_{\b}:=\{x\in \R^3: \b\le \sum_\theta|g_\theta(x)|^2*w_{M,\theta}\le 2\b\} .\] 
\end{thm}

\begin{proof} Assume throughout the proof that $U_{\b}= U_{\b}\cap B_R$. By the triangle inequality, it suffices to prove Theorem \ref{thm:proof-cone-dec} with $\Gamma$ replaced by the subset
\[ \Gamma_K=\{\xi\in\R^3:\xi_3 = \sqrt{\xi_1^2+\xi_2^2},\quad \arg(\xi_1,\xi_2)\le K^{-1},\quad 1\le \xi_3\le 1+K^{-1}\},\]
where we will choose $K>0$ later. Then the theorem holds with implicit constant the maximum of $K^{O(1)}$ and the implicit constant we get for $\Gamma_K$. Let $C(R)$ denote the smallest implicit constant so that \eqref{mainineq} is true for all $g$ with Fourier support in the $R^{-1}$ neighborhood of $\Gamma_K$ so that our goal is to show that $C(R)\le C_\e R^\e$.

We begin with some preliminary observations about subsets of $\Gamma_K$. Let $I=[-\d,\d]$. Define 
\[ \Gamma_K(I)=\{\xi\in\Gamma_K:\arg(\xi_1,\xi_2)\in I\}. \]
By Taylor expansion, the points in $\Gamma_K(I)$ have $\xi_3$-coordinate which may be expressed as
\[ \xi_3= \xi_1+\frac{1}{2}\xi_2^2+O(\xi_2^2|\xi_1-1|)=\xi_1+\frac{1}{2}\xi_2^2+O(\xi_2^2K^{-1}). \]
It follows that 
\[ \Gamma_K(I)\subset \mc{N}_{C\d^2K^{-1}}\left(\{\xi_1(1,0,1)+(0,\xi_2,\frac{1}{2}\xi_2^2):|\xi_2|\le 2\pi \d,|\xi_1|<10\}\right). \]
This is a cylindrical neighborhood of a $C\d^2 K^{-1}$-neighborhood of a $\sim \d$-arc in the parabola $(t,t^2/2)$. 

For Schwartz $g\colon\R^3\to\C$ with $\spt\,\widehat{g}\subset\mc{N}_{R^{-1}}(\Gamma_K)$, we have by Fourier inversion 
\begin{align*}
g(x)&=\int_{\mc{N}_{R^{-1}}(\Gamma_K)}e^{2\pi i x\cdot\xi}\widehat{g}(\xi)d\xi \\
    &=\int_{\mc{N}_{CK^{-3}}(\mc{P})} e^{2\pi i (x_2,x_3)\cdot \xi'}\int_{[0,3]}e^{2\pi i x\cdot\xi_1(1,0,1)}\widehat{g}(\xi_1(1,0,1)+(0,\xi'))d\xi_1d\xi' 
\end{align*}
where $\mc{P}=\{(t,t^2/2):|t|<1\}$. Write new variables $x=y_1\frac{1}{\sqrt{2}}(1,0,1)+y_2(0,1,0)+y_3\frac{1}{\sqrt{2}}(-1,0,1)$. Then after a change of variables, the previous displayed line equals
\begin{align*} 
{\sqrt{2}}\int_{\mc{N}_{CK^{-3}}(\mc{\tilde{P}})} &e^{2\pi i (y_2,y_3)\cdot (\xi_2,\xi_3)}e^{2\pi i y_1\xi_3}\\
&\times \int_{[0,3]}e^{2\pi i y_1 \xi_1\sqrt{2}}\widehat{g}(\xi_1(1,0,1)+(0,\xi_2,\sqrt{2}\xi_3))d\xi_1d\xi_2d\xi_3, \end{align*}
where $\tilde{\mathcal P}=\{(t,t^2/2\sqrt{2}):|t|<1\}$. For each $y_1$, define the function 
\[ \widehat{h_{y_1}}(\xi_2,\xi_3)=\sqrt{2}e^{2\pi i y_1\xi_3}\int_{[0,3]}e^{2\pi i y_1 \xi_1\sqrt{2}}\widehat{g}(\xi_1(1,0,1)+(0,\xi_2,\sqrt{2}\xi_3))d\xi_1.  \]
For cone planks $\tau \in\Theta(C^{-1}K^{-3/2})$ intersecting $\Gamma_K$, 
\[ \widehat{h_{y_1}}(\xi')=\sum_\tau \int_{[0,3]}e^{2\pi i y_1\xi_1\sqrt 2}\widehat{g_\tau}(\xi_1(1,0,1)+(0,\xi'))d\xi_1 \]
and each 
\[ \int_{[0,3]}e^{2\pi i y_1\xi_1\sqrt 2} \widehat{g_\tau}(\xi_1(1,0,1)+(0,\xi'))d\xi_1 \]
is supported in a $K^{-3}$-neighborhood of a $\sim K^{-3/2}$-arc of $\mc{\tilde{P}}$. 

Repeat the pigeonholing steps in the proof of Theorem \ref{Prefdec}. As in that proof, each step involves a negligible additive error, which we will omit here for simplicity. Let $\sum_b \phi_b(y_2,y_3)\sim 1$ on $U_{\b}$, where each $b$ is a 2-dimensional ball of radius $K$ in the $(y_2,y_3)$ variables. Assume that each $\widehat{\phi_b}$ (2-dimensional Fourier transform here) is supported in $B^{(2)}(0,K^{-1})$  and $\phi_b$ decays rapidly off of $b$. Then after two pigeonholing steps, 
\[ \int_{U_\beta}|\sum_\theta g_\theta|^6 \lesssim (\log R)^{12} \sum_b\int_{U_{\b}\cap (\R\times b)}|\sum_{\tau\in\mc{T}(b)} g_\tau|^6  \]
where for $ \b_0$ uniform across $b$ and $\tau$, 
\[ \int_{U_{\b} + (\{0\}\times  R^\e \tau_0^*)}|g_\tau|^6\lesssim (\log R)\int_{U_{ \b_0}(\tau)+ (\{0\}\times  R^\e \tau_0^*)}|g_\tau|^6, \]
in which $U_{ \b_0}(\tau)=\{y:\sum_{\theta\subset \tau}|g_\tau|^2*w_{M,\tau}(y)\sim\b_0\}$ and $\tau_0^*$ is the dual set in the $(y_2,y_3)$ variables to $\pi(\tau)$, where $\tau\subset [0,K^{-1}]_{y_1}\times \pi(\tau)_{(y_2,y_3)} $. 

For each $g_\tau$, let $\sum_{T\in\mathbb{T}_\tau} \s_T=1$ be a partition of unity localized to a tiling of $\R^3$ by dual planks $T$ to $\tau$. We may assume that each $T$ intersects the ball $B_R$. Each $\s_T$ has Fourier transform contained in $\tau$. We have the pointwise equality $g_\tau(x)=\sum_{T\in\mathbb{T}_\tau}g_\tau(x)\s_T(x)$ and inequality  $|g_\tau(x)|\lesssim \sum_{T\in\mathbb{T}_\tau}\|g_\tau \s_T^{1/2}\|_{L^\infty(T)}\s_T^{1/2}$. For each $\tau$, dyadically pigeonhole $\mathbb T_\tau'\subset\mathbb T_\tau$ so that 
\[ \sum_b\int_{U_{\b}\cap(\R\times b)}|\sum_{\tau\in\mc{T}(b)} g_\tau|^6\lesssim (\log R)^6\sum_b\int_{U_{\b}\cap(\R\times b)}|\sum_{\tau\in\mc{T}(b)} \sum_{T\in\mathbb{T}_\tau'} \s_{T}g_\tau|^6    \]
and $\|\s^{1/2}_Tg_\tau\|_{L^\infty(\R^3)}$ is comparable to a single value for all $T\in\mathbb T_\tau'$. Next for each $b$, pigeonhole to find $\b_K(b)>0$ so that 
\[ \int_{U_{\b}\cap(\R\times b)}|\sum_{\tau\in\mc{T}(b)} \sum_{T\in\mathbb{T}_\tau'} \s_{T}g_\tau|^6  \lesssim (\log R)\int_{U_{\b}\cap(\R\times b)\cap V_{\b_K}(b)}|\sum_{\tau\in\mc{T}(b)} \sum_{T\in\mathbb{T}_\tau'} \s_{T}g_\tau|^6  \]
where $V_{\b_K}(b)=\{y:\sum_{\tau\in\mc{T}(b)}|\sum_{T\in\mathbb{T}_\tau'}\s_Tg_\tau|^2*w_{M,\tau}(y)\sim\b_K(b)\}$. The final pigeonholing step is to refine the collection of $b$ so that $\b_K(b)$ is comparable to a single value $\b_K$. Assume from now on that we are only working with the final collection of pigeonholed $b$. 

It follows from Theorem \ref{Prefdec} and Fubini's theorem that
\[ \int_{U_\b}|\sum_\theta g_\theta|^6 \lesssim_\d K^\d \sum_b\int_{\pi_1(U_{\b}\cap V_{\b_K})}\left(\b_K^2\sum_{\tau\in\mc{T}} \int_{ R^\e b}|\sum_{T\in\mathbb{T}_\tau'}\s_T g_{\tau}\p_b|^2dy_2dy_3\right)dy_1   \]
where $\pi_1$ is orthogonal projection onto the span of $\frac{1}{\sqrt{2}}(1,0,1)$. 
Then by the pigeonholing of the wave packets $T$, for each $b$, the corresponding summand on the right-hand side above is bounded by 
\[ \int_{\pi(U_{\b}\cap V_{\b_K})} \left(\|\s_T^{1/2} g_\tau\|_\infty^{-4}\b_K^2\sum_{\tau\in\mc{T}(b)} \int_{ R^\e b}\|\s_T^{1/2}g_\tau\|_\infty^4|\sum_{T\in\mathbb{T}_\tau'}\s_T g_{\tau}\p_b|^2dy_2dy_3\right)dy_1\]
where $\|\s_T^{1/2}g_\tau\|_\infty$ is comparable to a single value. Note that
\[ \|\s_T^{1/2}g_\tau\|_\infty^4|\sum_{T\in\mathbb{T}_\tau'}\s_T g_{\tau}|^2\lesssim \sum_{T\in\mathbb{T}_\tau'}\|\s_T^{1/2}g_\tau\|_\infty^6\s_T\lesssim |g_\tau|^6*w_{N,\tau} \]
for any decay rate $N>0$. Using this and summing over $b$ leads to the expression 
\[ \sum_{\tau\in\mc{T}}\int_{U_{\b_0}(\tau)+(\{0\}\times R^\e\tau_0^*)}|g_\tau|^6.   \]
The summary of our inequality so far is that
\[ \int_{U_\b}|\sum_\theta g_\theta|^6 \lesssim_\d K^\d (\log R)^C\|\s_T^{1/2}g_\tau\|_\infty^{-4}\b_K^2\sum_{\tau\in\mc{T}}\int_{U_{ \b_0}(\tau)+(\{0\}\times R^\e \tau_0^*)}|g_\tau|^6. \]
Note that we replaced the sets $\mc{T}(b)$ by $\mc{T}$ in order to sum over all the $b$. 

Now use (cylindrical) parabolic rescaling and the definition of the constant $C(R)$ to bound each integral by 
\[\int_{U_{ \b_0}(\tau)+(\{0\}\times R^\e \tau_0^*)}|g_\tau|^6\le C(R/K)\b_0^2\sum_{\theta\subset\tau}\int_{\R^3}|g_\theta|^2 .  \]

We will show that 
\[ \|\s_T^{1/2}g_\tau\|_\infty^{-4}\b_K^2\b_0^2\lesssim \b^2. \]
By the pigeonholing procedure, there is some $b$ such that for some $y_b\in U_{\b}\cap (\R\times b)$,
\begin{align*} 
\b_K&\sim \sum_{\tau\in\mc{T}(b)}|\sum_{T\in\mathbb{T}_\tau'}\s_Tg_\tau|^2* w_{M,\tau_0}(y_b)
\end{align*}
and for each $\tau\in\mc{T}(b)$, there is $y_{b,\tau}\in U_{\a,\b}\cap(\R\times b)\cap \big(U_{ \b_0}(\tau)+(\{0\}\times R^\e \tau_0^*)\big)$ such that 
\[ \b_0\sim \sum_{\theta\subset\tau}|g_\theta|^2*w_{M,\theta}(y_{b,\tau}). \]
Since $y_b$ and each $y_{b,\tau}$ are contained in the same translate of $\{0\}\times R^\e \tau_0^*\subset\theta^*$, 
\[ |g_\theta|^2*w_{M,\theta}(y_{b,\tau})\lesssim |g_\theta|^2*w_{M,\theta}(y_b). \]
Since for all $y\in\R^3$ 
\[ \|\s_T^{1/2}g_\tau\|_{L^\infty(\R^3)}^{-2}|\sum_{T\in\mathbb{T}_\tau'}\s_Tg_\tau|^2*w_{M,\tau_0}(y)\lesssim \sum_{T\in\mathbb{T}_\tau'}\s_T*w_{M,\tau_0}(y)\lesssim 1,  \]
it follows that 
\[ \|\s_T^{1/2}g_\tau\|_\infty^{-4}\b_K^2\b_0^2\lesssim \b^2\lesssim \left(\sum_{\tau\in\mc{T}(b)}\sum_{\theta\subset\tau}|g_\theta|^2*w_{M,\theta}(y_b)\right)^2\sim\b^2. \]

Conclude that 
\[ C(R)\le C_\d (\log R)^C K^\d C(R/K), \]
where $K$ and $\d>0$ are to be chosen. Iterating $m$ times leads to 
\[ C(R)\le C_\d^m(\log R)^{Cm}K^{m\d}C(R/K^m). \]
Choose $m$ so $K^m\le R<K^{m+1}$. Then $C(R/K^m)\lesssim K^C$, so 
\[ C(R)\lesssim C_\d ^m(\log R)^{Cm}R^\d K^C. \]
Finally choose $K=R^\e$ and $\d=\e$ so that 
\[ C(R)\lesssim  C_\e^{2\e^{-1}} (\log R)^{C\e^{-1}} R^\e R^{C\e}. \]
Since $\log R\le \e^{-2}R^{\e^2}$, this concludes the proof. 
\end{proof}

\section{Sharpness}\label{sec:sharpness}
In this section, we show that Theorem \ref{thm:rect1} is sharp by constructing a random set of (almost) $\rho$-separated circles such that with high probability, there is only one value of $\mu$ such that every $1\times R^{1/2}\times R$-lightplank contains about $\mu$-many ``circles.'' The proof relies on standard large deviation estimates for sums of Bernoulli random variables, which we record in Appendix \ref{appendix}.

\begin{thm}\label{thm:sharp}
    For every $\varepsilon > 0$, the following holds for each $R\ge R_0(\varepsilon)$ sufficiently large and each $R^{\varepsilon} \le \rho \le R^{\frac12}$.
    
    With probability at least $.9$, there is a random set $X\subset[0,R]^3$ of cardinality $\sim R^{3+\varepsilon}\rho^{-3}$ such that every $\rho$-cube contains at most $R^\varepsilon$ points of $X$, and such that every $1\times R^{1/2}\times R$-lightplank $P$ in a maximal pairwise incomparable collection of such lightplanks contained in $[0,R]^3$ contains $\sim R^{3/2+\varepsilon}\rho^{-3}$-points of $X$. In particular,
    \[
    (R^{3/2+\varepsilon}\rho^{-3})^{4/3}|\mathcal P| \gtrapprox |X|^{4/3}.
    \]
\end{thm}
\begin{proof}
    We may assume $R > 10$. Let $\varepsilon>0$, and let $\rho \ge R^\varepsilon$, and consider the grid in $[0,R]^3$ formed by taking all closed cubes $\mathbf Q_i$ of side length $100\rho$. Within each cube $\mathbf Q_i$, choose a concentric sub-cube $Q_i$ of side length $\rho$. The resulting cubes $Q_i$ are at least $10\rho$-separated (being generous). Let $Y = \bigsqcup_{i}Q_i$, so $|Y| = 100^{-3}R^3$. We will choose points randomly and independently from $Y$ with probability $p = R^\varepsilon \rho^{-3}<1$, and let $X$ be the random set of points we obtain. By Proposition \ref{prop:large-deviations}, with probability at least $1-2e^{-\frac12|Y|p}$, $|X|\in [\frac{1}{10}|Y|p,10|Y|p]$.

    For each $i$, the expected cardinality of $X\cap Q_i$ is $\rho^3p$. Consider the event that $Q_i$ contains significantly more than its expected share of points of $X$; by Proposition \ref{prop:large-deviations},
    \begin{align*}
        \Prob(\frac{|X\cap Q_i|}{|Q_i|} \ge 10p) \le e^{-5|Q_i|p}=e^{-5R^{\varepsilon}}.
    \end{align*}
    The number of such events as $i$ ranges over the cubes $Q_i$ is $100^{-3}\rho^{-3}R^3$. Hence the probability that every cube $Q_i$ contains fewer than $10|Q_i|p$ points of $X$ is at least $1-100^{-3}\rho^{-3}R^3e^{-5R^{\varepsilon}}$, which is at least $.99$, provided $R$ is sufficiently large depending on $\varepsilon$. Hence, $X$ is a set with the property that every $\rho$-cube in $\R^3$ contains at most $10\rho^3p = 10R^\varepsilon$ points. It is plain to see from the proof of Theorem \ref{thm:main} that the proof continues to hold with a $R^{O(\varepsilon)}$-loss for such sets.

    Let $\mathcal P$ be a maximal pairwise incomparable collection of $1\times R^{1/2}\times R$-lightplanks contained in $[0,R]^3$. The cardinality of such a collection is $\sim R^2$. By Proposition \ref{prop:large-deviations}, for each $P\in \mathcal P$, the probability that $|X\cap P| \notin [\frac{1}{10}|P|p,10|P|p]$ is at most $2e^{-\frac{1}{2}|P|p}$. In particular, with probability at least $1-2|\mathcal P|e^{-\frac{1}{2}|P|p}$, every plank $P\in\mathcal P$ satisfies $|X\cap P| \in [\frac{1}{10}|P|p,10|P|p]$. The probability is at least $.99$, so long as $\rho\le R^{1/2}$ and $R$ is sufficiently large.
    
    By Theorem \ref{thm:main}, whose hypotheses hold with probability at least $.9$, with $\mu = R^{3/2+\varepsilon}{\rho^{-3}}$,
    \begin{align*}
        (R^{3/2+\varepsilon}{\rho^{-3}})^{4/3}|\mathcal P| \lesssim_\varepsilon R^\varepsilon |X|^{4/3} \sim R^\varepsilon\cdot R^{4+\frac{4}{3}\varepsilon}\rho^{-4}.
    \end{align*}

    On the other hand, by direct calculation,
    \[
    (R^{3/2+\varepsilon}{\rho^{-3}})^{4/3}|\mathcal P| = R^{4+\frac43\varepsilon}\rho^{-4}.\qedhere
    \]
\end{proof}

\newpage
\appendix
\section{Large deviations for sums of Bernoulli random variables}\label{appendix}
Here we record the estimates of large deviations for sums of Bernoulli random variables we need to establish the sharpness of Theorem \ref{thm:rect1}. These are standard applications of Chernoff bounds, but we include the proofs for completeness.
\begin{prop}\label{prop:large-deviations}
    Let $\{\xi_i:i=1,\dotsc,n\}$ be independent 0,1-valued random variables of mean $p$, and let $S_n = \sum_{i=1}^n\xi_i$. The following estimates hold
\begin{itemize}
\item[(i)] $p_n=\Prob(S_n>10np) \le e^{-5np}$
\item[(ii)] $q_n=\Prob(S_n<\frac{1}{10}np) \le e^{-\frac{np}{2}}$.
\end{itemize}
\end{prop}
\begin{proof}
First we prove (i). Let $t > 0$ be an extra parameter we have at our disposal. By Tchebychev's inequality and independence of the $\xi_i$,
\begin{align*}
p_n = \Prob[\sum_{i=1}^n\xi_i > 10np] &\le \Prob[e^{t\sum_{i=1}^n\xi_i}\ge e^{10tnp }] \\
&\le e^{-10tnp }\E[e^{t\sum_{i=1}^n\xi_i}] \\
&= e^{-10t np }\prod_{i=1}^n \E e^{t\xi_i}.
\end{align*}
A direct calculation shows
\[
\E e^{t\xi_i} = e^tp + (1-p) = 1 + (e^t-1)p.
\]
By what we have so far,
\[
p_n \le e^{-10tnp } [1+(e^t-1)p]^n \le e^{-10 t np } e^{(e^t-1)np} = [e^{-10t+(e^t-1)}]^{np}.
\]
We could optimize to choose the best value of $t$, but setting $t = 1$ is sufficient because it shows
\[
p_n \le [e^{-8.28...}]^{np}\le e^{-5np}.
\]

The proof of (ii) is very similar.  Let $\eta_i = p-\xi_i$, and let $T_n = \sum_{i=1}^n \eta_i$. Then the reader can easily verify
\[
\Prob(S_n<\frac{1}{10}np) = \Prob(T_n > \frac{9}{10}np).
\]
By Tchebychev and independence of the variables $\eta_i$:
\begin{align*}
\Prob(T_n>\frac{9}{10}np) &\le e^{-\frac{9}{10}tnp}\prod_{i=1}^n\mathbb E e^{t\eta_i}\\
&= e^{-\frac{9}{10}tnp}(e^{tp}(1-p)+e^{t(p-1)}p)^n\\
&= e^{\frac{1}{10}tnp}(1+(e^{-t}-1)p)^n\\
&\le e^{\frac{1}{10}tnp}e^{(e^{-t}-1)np} = [e^{\frac{t}{10}+(e^{-t}-1)}]^{np}.
\end{align*}
Setting $t = 1$, we have
\[
q_n = \Prob(T_n>\frac{9}{10}np) \le [e^{-0.532...}]^{np} \le e^{-\frac{np}{2}}.\qedhere
\]
\end{proof}

\bibliographystyle{amsplain}
\bibliography{main}

\begin{bibdiv}
\begin{biblist}

\bib{besicovitch1968plane}{article}{
      author={Besicovitch, AS},
      author={Roda, R},
       title={A plane set of measure zero containing circumferences of every radius},
        date={1968},
     journal={Journal of the London Mathematical Society},
      volume={1},
      number={1},
       pages={717\ndash 719},
}

\bib{bourgain2015proof}{article}{
      author={Bourgain, Jean},
      author={Demeter, Ciprian},
       title={The proof of the {$l^2$} decoupling conjecture},
        date={2015},
     journal={Annals of Mathematics},
       pages={351\ndash 389},
}

\bib{burkholder1972integral}{article}{
      author={Burkholder, D.L.},
      author={Davis, B.J.},
      author={Gundy, R.F.},
       title={Integral inequalities for convex functions of operators on martingales},
        date={1972},
     journal={Proceedings of the Sixth Berkeley Symposium on Mathematical Statistics and Probability, Volume 2: Probability Theory},
       pages={223\ndash 240},
}

\bib{clarkson1990combinatorial}{article}{
      author={Clarkson, Kenneth~L},
      author={Edelsbrunner, Herbert},
      author={Guibas, Leonidas~J},
      author={Sharir, Micha},
      author={Welzl, Emo},
       title={Combinatorial complexity bounds for arrangements of curves and spheres},
        date={1990},
     journal={Discrete \& Computational Geometry},
      volume={5},
      number={2},
       pages={99\ndash 160},
}

\bib{cohen2025lower}{article}{
      author={Cohen, Alex},
      author={Pohoata, Cosmin},
      author={Zakharov, Dmitrii},
       title={Lower bounds for incidences},
        date={2025},
     journal={Inventiones mathematicae},
       pages={1\ndash 74},
}

\bib{du2018pointwise}{inproceedings}{
      author={Du, Xiumin},
      author={Guth, Larry},
      author={Li, Xiaochun},
      author={Zhang, Ruixiang},
       title={Pointwise convergence of {S}chr{\"o}dinger solutions and multilinear refined {S}trichartz estimates},
organization={Cambridge University Press},
        date={2018},
   booktitle={Forum of mathematics, sigma},
      volume={6},
       pages={e14},
}

\bib{ellenberg2016new}{article}{
      author={Ellenberg, Jordan~S},
      author={Solymosi, Jozsef},
      author={Zahl, Joshua},
       title={New bounds on curve tangencies and orthogonalities},
        date={2016},
     journal={Discrete Analysis},
}

\bib{erdos1960sets}{article}{
      author={Erd\H{o}s, P.},
       title={On sets of distances of $n$ points in {E}uclidean space},
        date={1960},
     journal={Magyar Tud. Akad. Mat. Kut. Int. K{\"o}zl. (Publ. Math. Inst. Hung. Acad. Sci.)},
      volume={5},
       pages={165\ndash 169},
}

\bib{fu2023decoupling}{article}{
      author={Fu, Yuqiu},
      author={Guth, Larry},
      author={Maldague, Dominique},
       title={Decoupling inequalities for short generalized dirichlet sequences},
        date={2023},
     journal={Analysis \& PDE},
      volume={16},
      number={10},
       pages={2401\ndash 2464},
}

\bib{gan2022restricted}{article}{
      author={Gan, Shengwen},
      author={Guo, Shaoming},
      author={Guth, Larry},
      author={Harris, Terence~LJ},
      author={Maldague, Dominique},
      author={Wang, Hong},
       title={On restricted projections to planes in {$\mathbb R^3$}},
        date={2022},
     journal={arXiv preprint arXiv:2207.13844},
}

\bib{gan2024exceptional}{article}{
      author={Gan, Shengwen},
      author={Guth, Larry},
      author={Maldague, Dominique},
       title={An exceptional set estimate for restricted projections to lines in {$\mathbb R^3$}},
        date={2024},
     journal={The Journal of Geometric Analysis},
      volume={34},
      number={1},
       pages={15},
}

\bib{guth2020falconer}{article}{
      author={Guth, Larry},
      author={Iosevich, Alex},
      author={Ou, Yumeng},
      author={Wang, Hong},
       title={On {F}alconer’s distance set problem in the plane},
        date={2020},
     journal={Inventiones mathematicae},
      volume={219},
      number={3},
       pages={779\ndash 830},
}

\bib{guth2021improved}{article}{
      author={Guth, Larry},
      author={Maldague, Dominique},
      author={Wang, Hong},
       title={Improved decoupling for the parabola},
        date={2021},
     journal={Journal of the European Mathematical Society},
      volume={26},
      number={3},
       pages={875\ndash 917},
}

\bib{guth2019incidence}{article}{
      author={Guth, Larry},
      author={Solomon, Noam},
      author={Wang, Hong},
       title={Incidence estimates for well spaced tubes},
        date={2019},
     journal={Geometric and Functional Analysis},
      volume={29},
      number={6},
       pages={1844\ndash 1863},
}

\bib{harris2024projections}{incollection}{
      author={Harris, Terence L.~J.},
       title={Length of sets under restricted families of projections onto lines},
        date={2024},
   booktitle={Recent developments in harmonic analysis and its applications},
      editor={Guo, Shaoming},
      editor={Li, Zane~Kun},
      editor={Street, Brian},
      series={Contemporary Mathematics},
      volume={792},
   publisher={American Mathematical Society},
     address={Providence, RI},
       pages={1\ndash 17},
}

\bib{kinney1968thin}{article}{
      author={Kinney, JR},
       title={A thin set of circles},
        date={1968},
     journal={The American Mathematical Monthly},
      volume={75},
      number={10},
       pages={1077\ndash 1081},
}

\bib{ortiz2024sharp}{article}{
      author={Ortiz, Alexander},
       title={A sharp weighted {F}ourier extension estimate for the cone in {$\mathbb{R}^3$} based on circle tangencies},
        date={2025},
     journal={Journal d'Analyse Mathématique},
        note={To appear. Preprint available at \href{https://arxiv.org/abs/2307.11731}{arXiv:2307.11731}},
}

\bib{ren2023furstenberg}{article}{
      author={Ren, Kevin},
      author={Wang, Hong},
       title={Furstenberg sets estimate in the plane},
        date={2023},
     journal={arXiv preprint arXiv:2308.08819},
}

\bib{schlag1996thesis}{thesis}{
      author={Schlag, Wilhelm},
       title={{$L^p \to L^q$ {E}stimates for the {C}ircular {M}aximal {F}unction}},
        type={Ph.D. Thesis},
     address={Pasadena, CA},
        date={1996},
}

\bib{schlag2003continuum}{article}{
      author={Schlag, Wilhelm},
       title={On continuum incidence problems related to harmonic analysis},
        date={2003},
     journal={Journal of Functional Analysis},
      volume={201},
      number={2},
       pages={480\ndash 521},
}

\bib{wolff1999recent}{article}{
      author={Wolff, Thomas},
       title={Recent work connected with the {K}akeya problem},
        date={1999},
     journal={Prospects in mathematics (Princeton, NJ, 1996)},
      volume={2},
      number={129-162},
       pages={4},
}

\bib{wolff2000local}{article}{
      author={Wolff, Thomas},
       title={Local smoothing type estimates on {$L^p$} for large {$p$}},
        date={2000},
     journal={Geometric and Functional Analysis},
      volume={10},
      number={5},
       pages={1237\ndash 1288},
}

\bib{zahl2019breaking}{article}{
      author={Zahl, Joshua},
       title={Breaking the 3/2 barrier for unit distances in three dimensions},
        date={2019},
     journal={International Mathematics Research Notices},
      volume={2019},
      number={20},
       pages={6235\ndash 6284},
}

\end{biblist}
\end{bibdiv}

\end{document}